\documentclass[aap,preprint]{imsart}
\RequirePackage{amsthm,amsmath,amsfonts,amssymb}
\RequirePackage[numbers]{natbib}

\usepackage{amsmath,amssymb,mathrsfs,amsthm,amsfonts,bbm, dsfont}
\usepackage{times}
\usepackage{enumerate}
\usepackage{epsfig}
\usepackage{graphicx,subfigure}
\usepackage[section]{placeins}
\usepackage{url}
\usepackage{stmaryrd}
\usepackage{multido}
\usepackage{tikz}
\usepackage{pgf,wrapfig}
\usepackage{titlesec}

\startlocaldefs

\numberwithin{equation}{section}

\newtheorem{lemma}{Lemma}[section]
\newtheorem{theorem}{Theorem}

\newtheorem{coro}{Corollary}[section]
\newtheorem{prop}{Proposition}[section]

\newtheorem{assumption}{Assumption}

\newenvironment{thmstar}[1]
{\renewcommand{\thetheorem}{\ref{#1}${\bf {^*}}$}%
   \addtocounter{theorem}{-1}%
   \begin{theorem}}
  {\end{theorem}}

\newcommand{\RR}{{\mathbb R}}
\newcommand{\ZZ}{{\mathbb Z}}
\newcommand{\NN}{{\mathbb N}}

\newcommand{\lm}{\lambda}
\newcommand{\ep}{\epsilon}

\newcommand{\ra}{\rightarrow}
\newcommand{\deq}{\stackrel{\rm d}{=}}

\newcommand{\qandq}{\quad\mbox{and }}
\newcommand{\qifq}{\quad\mbox{if }}

\newcommand{\qforq}{\quad\mbox{for }}

\newcommand{\qasq}{\quad\mbox{as ~}}
\newcommand{\qinq}{\quad\mbox{in }}
\newcommand{\qforallq}{\quad\mbox{for all }}

\def\tinf{\rightarrow\infty}

\def\Ra{\Rightarrow}

\newcommand{\bes}{\begin{equation*}}
\newcommand{\ees}{\end{equation*}}
\newcommand{\bequ}{\begin{equation}}
\newcommand{\eeq}{\end{equation}}
\newcommand{\bsplit}{\begin{split}}
\newcommand{\esplit}{\end{split}}
\newcommand{\bea}{\begin{eqnarray}}
\newcommand{\eea}{\end{eqnarray}}
\newcommand{\beas}{\begin{eqnarray*}}
\newcommand{\eeas}{\end{eqnarray*}}
\newcommand{\benum}{\begin{enumerate}}
\newcommand{\eenum}{\end{enumerate}}

\newcommand{\wt}{\widetilde}

\newcommand{\wh}{\widehat}
\newcommand{\UU}{\mathcal{U}}
\newcommand{\PP}{\mathcal{P}}

\newcommand{\D}{ D}

\endlocaldefs

\begin{document}

\begin{frontmatter}
\title{Many-Server Heavy-Traffic Limits for Queueing Systems with Perfectly Correlated Service and Patience Times}
\runtitle{Systems with Perfectly Correlated Service and Patience}

\begin{aug}
  \author[A]{\fnms{Lun} \snm{Yu}\ead[label=e1]{lunyu2014@u.northwestern.edu}}\and
  \author[A]{\fnms{Ohad} \snm{Perry}\ead[label=e2]{ohad.perry@northwestern.edu}}
  \address[A]{Department of Industrial Engineering and Management Sciences,
  2145 Sheridan Rd. Evanston, IL 60208,\printead{e1,e2}}
\end{aug}

\begin{abstract}
We characterize heavy-traffic process and steady-state limits for systems staffed according to the square-root safety rule,
when the service requirements of the customers are perfectly correlated with their individual patience for waiting in queue.
Under the usual many-server diffusion scaling, we show that the system is asymptotically equivalent to a system with no abandonment.
In particular, the limit is the Halfin-Whitt diffusion for the $M/M/n$ queue
when the traffic intensity approaches its critical value $1$ from below,
and is otherwise a transient diffusion, despite the fact that the prelimit is positive recurrent.
To obtain a refined measure of the congestion due to the correlation,
we characterize a lower-order fluid (LOF) limit for the case in which the diffusion limit is transient,
demonstrating that the queue in this case scales like $n^{3/4}$.
Under both the diffusion and LOF scalings, we show that the stationary distributions converge weakly to the time-limiting behavior
of the corresponding process limit.
\end{abstract}



\end{frontmatter}


\section{Introduction} \label{secIntro}

Service systems often experience abandonment due to customer impatience for waiting in queue.
The significant impacts that abandonment has on the queueing dynamics are clear from the fact that stability---
the most fundamental performance measure of a queueing system---is guaranteed to hold under weak regularity conditions on the system's primitives,
regardless of the value of the traffic intensity; see \cite[\S 4]{kang2012asymptotic}.
To model customer abandonment, it is typically assumed that the patience of the customers are independent and identically distributed (i.i.d.) random variables,
that are also independent of all other random variables and processes in the model.
However, 
it stands to reason that, in practice, the patience of customers depends on their individual service requirement, as was indeed
empirically demonstrated to be the case in contact centers \cite{reich2012offered} and restaurants \cite{de2017worth}.

A heuristic fluid model developed in \cite{wu2018service} suggests that positive dependence between the service and patience times of customers
have large impacts on steady-state performance measures, such as the expected steady-state queue length and waiting times, when the system
is overloaded (in the sense that the arrival rate exceeds the maximum service capacity).
However, in overloaded systems, practically all the customers are delayed in queue,
and their waiting times are, asymptotically (under fluid scaling), of the same order as the service time.
It is therefore not immediately clear whether the insights in
\cite{wu2018service} extend to systems that are not overloaded, so that a significant proportion of the customers are not delayed at all,
and the waiting times of those customers that are delayed are asymptotically negligible.

In this paper, we carry out asymptotic analysis in this latter setting, by considering systems that are staffed according to the square-root rule,
whose aim is to put the systems in the Halfin-Whitt limiting regime. This regime, which was first
characterized in the seminal paper \cite{halfin1981heavy} for the $M/M/n$ (Erlang-C) queue,
and was later extended in \cite{garnett2002designing} to the $M/M/n + M$ (Erlang-A) model,
which includes exponentially distributed customer patience,
is also known as the {\em quality-and-efficiency} (QED) regime, as it achieves both efficient utilization,
while simultaneously providing high quality of service.
In particular, under standard independence assumptions of the system's primitives,
the square-root staffing rule guarantees that almost all the service capacity is utilized at all times,
as is the case in the conventional heavy traffic regime, yet the probability that arrivals are delayed in queue is smaller than $1$ in the limit,
and waiting times of delayed customers are asymptotically negligible; see, e.g., \cite{van2019economies, whitt1992understanding}.
It is significant that the Erlang-A model operates in the QED regime even if the traffic intensity approaches $1$ from above, namely, if the
service capacity in the system is smaller than the demand for service by an $O(\sqrt{n})$ term. We elaborate in Section \ref{secGeneral} below.

{\em The Impact of the Correlation.}
When the service and patience times are positively correlated, one expects the system to be more congested than
when the two times are independent, because delayed customers that do not abandon tend to spend more time in service than a ``generic'' customer.
On the other hand, the waiting times and the proportion of abandonment in the QED regime are asymptotically negligible, and so the extent to which
correlation impacts the queueing dynamics is not {\it a priori} clear.

Our results show that, in the perfect-correlation case, abandonment has an asymptotically diminishing impact on the queues under diffusion scaling,
in that the system behaves much like a system that has no abandonment at all. Thus, unlike in the typical ``independent models''
(which assume that all the primitive processes are mutually independent), the diffusion limit can be transient, despite the fact that the prelimit is always stable.
The exact extent to which the correlation impacts congestion follows from limits for the queue process and for its steady-state distribution that
are achieved under an $n^{3/4}$ spatial scaling.

Specifically, we prove the following functional weak limit theorems.
The diffusion limit,
which is achieved under the usual many-server diffusion scaling (see Section \ref{secModel2}),
is the same limit that is obtained for the Erlang-C model under the square-root staffing rule.
Thus, if the traffic intensity approaches $1$ from below as $n \tinf$, then the diffusion limit
is the Halfin-Whitt diffusion in \cite{halfin1981heavy}. On the other hand, if the traffic intensity approaches $1$ from above,
then the limit is a transient diffusion, having a positive drift.
To obtain the exact order of congestion in the latter case, we derive a lower-order fluid (LOF) limit,
and a corresponding weak limit for the stationary distributions, both obtained under spatial scaling of $n^{3/4}$.
Given that the Erlang-A model operates in the QED regime under the square root staffing rule,
those latter limit theorems imply that the correlation causes an increase of order $O(n^{1/4})$ in congestion relative to the independent case.

{\em Implications.}
Even though perfect correlation between the service and patience times of customers is unlikely to be encountered in practice,
this case is worth studying because the limits we obtain for the queues are simple one-dimensional Markov processes that are easy to interpret,
despite the non-Markovian nature of the prelimit queue.
More general dependence structure will necessarily require complex (e.g., measured-valued) process descriptors,
which will in turn lead to more complex, infinite-dimensional limiting processes; see \cite{puha2019scheduling} for background.
On the other hand, the diminishing impact of the abandonment on the system's dynamics, and the resulting congestion,
are likely to hold in much greater generality than the special case we study. (In a similar vein to the QED regime,
which was initially developed for systems with exponentially distributed service times,
and was only later shown to hold in greater generality \cite{gamarnik2013steady, puhalskiiReiman2000multiclass, reed2009g}.)

We further remark that certain a martingale property, that is key to deriving measure-valued limits for a non-Markovian many-server queues with abandonment,
relies heavily on having the service and patience times be independent; see \cite[Proposition 5.1]{kang2010fluid}.
In the special case we consider, we circumvent this issue by employing an intricate representation of the state descriptors,
exploiting sub-martingale properties of certain two-parameter processes.
See the state descriptors in Section \ref{secMartRep} and Lemma \ref{Lem:FMartingales} below.


\subsection{Background and Related Literature}
Consider a sequence of systems, in which the $n$th element has a pool of $n$ statistically
homogeneous agents serving a single class of statistically homogeneous customers.
Let $\lm^n$ denote the arrival rate to system $n$ and $\mu$ denote the service rate of a customer (the latter does not scale with the system).
The square-root staffing rule stipulates that the number of agents and the arrival rate satisfy the relation $\lm^n = n\mu - O(\sqrt{n})$ as $n \tinf$,
while, simultaneously, $\lm^n/n \ra \lm$, for some $\lm  > 0$.
Equivalently, the square-root rule implies that, for some $\beta > 0$,
\bequ \label{Square root}
\lim_{n\tinf} \sqrt{n}(1-\rho^n) = \beta,
\eeq
where $\rho^n := \lm^n/(n\mu)$ is the traffic intensity to system $n$.

Now, consider the special case of Poisson arrivals and exponentially distributed service times, namely, the Erlang-C queue.
Let $X^n_C := \{X^n_C(t) : t \ge 0\}$ denote the number-in-system process, and let $\wh X^n_C := \{\wh X^n_C(t) : t \ge 0\}$ denote its diffusion-scale version;
\bes
\wh X^n_C(t) := n^{-1/2}(X^n_C(t) - n), \quad t \ge 0.
\ees
Let $\Ra$ denote convergence in distribution.
Theorem 2 in \cite{halfin1981heavy} states that, if \eqref{Square root} holds, and in addition
$\wh X^n_C(0) \Ra X_0$ in $\RR$, then $\wh X^n_C \Ra \wh X_C$ uniformly on compact intervals as $n\tinf$,
where $\wh X_C := \{\wh X_C(t) : t \ge 0\}$ is the unique strong solution (e.g., see \cite{revuz2013continuous}) to the Stochastic Differential Equation (SDE)
\bequ \label{HWdiff}
d \wh X_C(t) = m_C(\wh X_C(t)) dt + \sqrt{2 \mu} dB(t),\;\; \wh X_C(0)=X_0,
\eeq
for 
\bes
m_C(x) :=
\left\{\begin{array}{ll}
-\mu \beta & \text{if } x \ge 0; \\
-\mu (\beta + x) & \text{if } x < 0,
\end{array} \right.
\ees
and $B := \{B(t) : t \ge 0\}$ denoting a standard Brownian motion.

If in addition, customers are assumed to have finite patience that is exponentially distributed with mean $1/\theta$
that is independent of all other random variables in the model,
namely, if the Erlang-A queue is considered, then the square-root staffing rule can be generalized by allowing $\beta$ in \eqref{Square root} to be negative.
In particular, let $X_A^n(t)$ denote the number-in-system process in a system with abandonment, and let
$$\wh X_A(t) := n^{-1/2}(\wh X^n_A(t) - n), \quad t \ge 0.$$
Theorem 2 in \cite{garnett2002designing} proves that, if \eqref{Square root} holds with $\beta \in (-\infty, \infty)$, and in addition, $\wh X_A^n(0) \Ra X_0$ in $\RR$,
for some proper random variable $X_0$, then $\wh X^n_A \Ra \wh X_A$ uniformly over compacts as $n\tinf$, where
\bequ \label{GarnettDiff}
d \wh X_A(t) = m_A(\wh X_A(t)) dt + \sqrt{2 \mu} dB(t),\;\; \wh X_A(0)=X_0.
\eeq
Here, $B$ denotes a standard Brownian motion as before, and
\bes
m_A(x) :=
\left\{\begin{array}{ll}
-(\mu \beta + \theta x) & \text{if } x \ge 0; \\
-\mu (\beta + x) & \text{if } x < 0.
\end{array} \right.
\ees

We observe that both the diffusion limit in \eqref{HWdiff} and the limit in \eqref{GarnettDiff} imply that the stochastic fluctuations of $X^n_C$ and $X^n_A$ about $n$
(the number of agents) are of order $\sqrt{n}$, which we denote by $O_P(\sqrt{n})$.
Therefore, both the number of idle agents and the number of customers waiting in queue are $O_P(\sqrt{n})$ as well, as $n\tinf$.
Moreover, both diffusion processes achieve values in $\RR$, implying that a nonnegligible proportion of the customers do not wait at all,
while the waiting times of those customers who are delayed in queue are $O_P(n^{-1/2})$, and so are asymptotically negligible, as $n\tinf$.

\subsection{Notation} \label{secNotation}
All the random elements are defined on a complete probability space $(\Omega, \mathcal F, P)$; expectation with respect to $P$ is denoted by $E$.
We let $\RR$ and $\ZZ$ denote the sets of real numbers and integers, respectively, with $\RR_+ := [0,\infty)$ and $\ZZ_+ := \ZZ \cup \RR_+$.
For $k \in \NN$, we let $\RR^k$ denote the space of $k$-dimensional vectors with real components.
We let $D^k$ denote the space of right-continuous $\RR^k$-valued functions (on arbitrary finite time intervals) with limits everywhere,
endowed with the usual Skorokhod $J_1$ topology; see \cite{billingsley2009convergence}.
We let $D := D^1$ and $D_0:=\{x\in D:x(0)\ge0\}$.
We use $C^k$ (and $C := C^1$) to denote the subspace of $D^k$ of continuous functions, and $C_0:=D_0\cap C$.
It is well-known that the $J_1$ topology relativized to $C^k$ coincides with the uniform topology on $C^k$, which is induced by the norm
\begin{equation*}
||x||_t := \sup_{0 \leq u \leq t} \|x(u)\|,
\end{equation*}
where $||x||$ denotes the usual Euclidean norm of $x \in \RR^k$. We use $\eta : \RR \ra \RR$ for the identity map, i.e., $\eta(t)=t$ for $t\ge0$.

For a sequence of processes $\{Y^n : n \ge 1\}$ and a sequence of scalar $\{a^n : n \ge 1\}$,
we write (i) $Y^n=o_P(a^n)$, if for any $t \ge 0$ we have $\|Y^n/a^n\|_t\Ra0$ in $\RR$, as $n\tinf$;
(ii) $Y^n=O_P(a^n)$, if $Y^n$ is stochastically bounded, i.e., $\{\|Y^n/a^n\|_t:n\in\ZZ_+\}$ is a tight sequence in $\RR$ for any $t\ge 0$;
(iii) $Y^n = \Theta_p(a^n)$ if $Y^n=O_P(a^n)$ but not $o_P(a^n)$.
We write $\deq$ to denote equality in distribution, and $\le_{s.t.}$ to denote the usual stochastic order, namely, for two random variables $X$ and $Y$,
we write $X\le_{s.t.} Y$ if
$P(X>x)\le P(Y>x)$ for all $x \in \RR$. For a random variable with values in $ZZ_+$, and a sequence of random variables $\{X_i : i \ge 1\}$,
we define $\sum_{i=1}^{N} X_n := 0$ on the event $\{N = 0\}$.

We denote $x^+:=\max\{x,0\}$ and $x^-:=-\min\{x.0\}$ for $x\in\RR$.
For $x$, $y\in\RR$ we let $x\wedge y:=\min\{x,y\}$ and $x\vee y:=\max\{x,y\}$.
Moreover, we let the latter ``min'' and ``max'' operators $\wedge$ and $\vee$ have higher precedence than multiplication, so that $xy\wedge z=x(y\wedge z)$,
and in particular, $x+y\wedge z = x+(y\wedge z)$, for $x, y, z \in \RR$.

\subsection{Organization}
The rest of the paper is organized as follows: We introduce the model in Section \ref{secModel}.
The main results, namely, the diffusion and LOF limits, as well as the corresponding weak limits for the stationary distributions, appear in Section \ref{secMain}.
To simplify the exposition, we first introduce the stochastic-process limit theorems under a simplifying assumption on the initial conditions;
we weaken that assumption significantly in Section \ref{secGeneral}.
In Section \ref{secTransient} we provide a characterization of the system's dynamics that is key to establishing the main results,
whose proofs appear in Section \ref{secMainProof}.
Proofs of supporting results are given in Sections \ref{secAuxProof1}--\ref{SecProofsLemmas}.

\section{The Model} \label{secModel}

We consider a sequence of systems denoted by $M/M_{pc}/n + M_{pc}$, indexed by the number of agents $n$; the subscript `pc' is mnemonic for ``perfect correlation.''
Each of the systems along the sequence consists of a single service pool with statistically homogeneous agents,
and an infinite buffer in which customers wait for their service.
Customers arrive to system $n$ according to a Poisson process with rate $\lm^n$, where $\lm^n/n \ra \lm$ as $m \tinf$, for some $\lm > 0$.
A customer begins service with an agent immediately upon arrival, if an idle agent is available,
and otherwise, waits in the queue for his turn to enter service.
We assume that customers are served in accordance with the FIFO discipline, namely, in the order of arrival, and that each customer has finite patience for waiting in queue.
That is, if a customer runs out of patience before his turn to enter service, that customer abandons the queue without returning.
We further assume that the service requirement and the patience time of each customer are (marginally) exponentially distributed with respective means
$1/\mu$ and $1/\theta$, $\mu, \theta > 0$, and that these two exponential random variables are independent from the arrival process and from the
service and patience times of all other customers. Without loss of generality, we measure time in service-time units, taking $\mu = 1$.
We further assume that $n$, $\lm^n$ and $\mu$ are related via the limit \eqref{Square root} (so that $\lm = \mu =1$),
for some $\beta \in (-\infty, \infty)$.

Unlike the standard $M/M/n+M$ queue, we assume that the service requirement of a customer
is perfectly correlated with his patience. In particular, Let $(S, T)$ denote a random variable in $\RR^2$,
such that $T$ is exponentially distributed with mean $1/\theta$ and $S$ is exponentially distributed
with mean $1/\mu = 1$. The assumption that $S$ and $T$ are perfectly correlated implies that $T = S/\theta$ w.p.1.
We assume that the service requirement and patience of each customer is a draw from the joint distribution of $S$ and $T$,
independently of all other customers and of the arrival process.

Due to the assumed correlation, the service-time distribution of a {\em served customer} is different than the service-time distribution of a generic customer.
For $w \ge 0$, let $S'(w)$ denote a generic random variable distributed like the service time of a customer who waited $w$ time units in queue.
Utilizing the memoryless property of the exponential distribution, we have that
\begin{equation} \label{servTime}
  S'(w) \deq S_b+\theta w,
\end{equation}
where $S_b$ is an exponentially distributed random variable with mean $1$.
Thus, the service time of each customer can be thought of as having two independent phases: conditional on the waiting time of the customer being $w$,
phase $1$ takes $\theta w$ units of time, and Phase $2$ is distributed like $S_b$. In particular, the waiting time in queue completely determines
the length of phase $1$, so that, conditional on his waiting time, the service time of a customer is a shifted exponential random variable.

For $t \ge 0$ and $n \ge 1$, let $Z^n(t)$ denote the the number of customers in service at time $t$,
and let $Z^n_i(t)$ denote the number of customers in phase $i$ at time $t$, $i = 1,2$, so that
\begin{equation*}
Z^n=Z_1^n+Z_2^n.
\end{equation*}
We denote by $Q^n(t)$ the number of customers waiting in queue, and by $X^n(t)$ the total number of customers in the system at time $t$, so that
$X^n(t) := Z^n(t) + Q^n(t)$.

\subsection{Preliminary: Stationarity of the $M/M_{pc}/n+M_{pc}$ System}

\begin{theorem} \label{ThErgodicity}
$X^n$ possesses a unique steady-state distribution, which is also its limiting distribution, as $t\tinf$.
\end{theorem}

\begin{proof}
  First note that, due to the arrival process being Poisson, and the fact that all customers entering service immediately upon arrival have i.i.d.\ exponential service times, $X^n$ is a regenerative process, with state $0$ being a regeneration point.
  By \cite[Theorem 2.1(b)]{sigman1993review}, we only need to demonstrate that $X^n$ is a positive recurrent regenerative process.
  We prove this result by bounding the sample paths of $X^n$ from above with a positive recurrent process via coupling the $M/M_{pc}/n+M_{pc}$ with
  an infinite-server queue.
To this end, we give the two systems the same initial number of customers, and the same Poisson arrival process,
letting the service time of each arrival to the infinite-server queue
be equal to the service plus patience time of the corresponding customer in the $M/M_{pc}/n+M_{pc}$ system.
In particular, with $(S_i, T_i)$ denoting the service-patience times bivariate corresponding to the the $i$th arrival to the $M/M_{pc}/n+M_{pc}$ system,
we take $S_i+T_i$ to be the service time of the same arrival to the infinite-server system. Note that $S_i+T_i$
is exponentially distributed with rate $\theta/(1+\theta)$ because $S_i = \theta T_i$.

If $X^n(0) = K > 0$, then we endow each ``initial'' customer $k$, $1 \le k \le K$, with a bivariate $(S_k, T_k)$,
such that $S_k$ is exponentially distributed with mean $1$, $T_k = S_k/\theta$,
and these $K$ bivariates are i.i.d. We let the remaining service time of each such customer $k$ in the infinite-server queue
be $S_k + T_k$ (so that it is exponentially distributed with rate $\theta(1 + \theta)^{-1}$), and the remaining service time in the $M/M_{pc}/n+M_{pc}$ system
be an arbitrary number that is no larger than $S_k$; the remaining time to abandon of customer $k$ that is waiting in
the $M/M_{pc}/n+M_{pc}$ queue is no larger than $T_k$.

Under the above construction, the infinite-server queue is an $M/M/\infty$ system, and in particular, a CTMC.
Since the time that a customer with patience $T$ and service requirement $S$ spends in the
$M/M_{pc}/n+M_{pc}$ is smaller than $S+T$ w.p.1, the $k$th ``initial customer'' and the $i$th arrival after time $0$
depart the $M/M_{pc}/n+M_{pc}$ system before they depart the infinite-server system,
implying that the sample path of the queue in the latter system is (weakly) larger than in the former w.p.1.
In turn, whenever the $M/M/\infty$ system is empty, so is the $M/M_{pc}/n+M_{pc}$ system.
Now, the $M/M/\infty$ queue is an ergodic continuous-time Markov chain, regardless of the values of the arrival and service rates, and so its expected
busy cycle is finite.
This immediately implies that the regenerative cycle length, namely, the time between two consecutive visits to the empty state, is finite in the $M/M_{pc}/n+M_{pc}$ system.
\end{proof}

Henceforth, we let $X^n(\infty)$ denote a random variable having the unique stationary (and limiting) distribution of the process $X^n$.

\section{Main Results} \label{secMain}
In this section we present the main results of the paper, namely the diffusion and LOF limit, and the corresponding weak limits for the stationary distributions.
Throughout, we assume that \eqref{Square root} holds; the specific range of values that $\beta$ achieves
is specified in the formal statements.

\subsection{Limit Theorems Under Diffusion Scaling} \label{secModel2}
The diffusion limit is achieved under the usual many-server diffusion scaling for the scaled number-in-system process
 $$\wh X^n := n^{-1/2}(X^n - n).$$
 We note that, since $X^n$ is not a Markov process, the value of $X^n(0)$ does not determine the law of $X^n$.
Nevertheless, we can characterize the dynamics of $X^n$ without resorting to infinite-dimensional (measure-valued) Markov representation
for a special class of natural initial conditions. In particular, we can consider the case in which
the system has started operating before time $0$, such that all of the customers at time $0$ are in service, and none of them experienced any wait before entering service.
(For example, the system can be initialized empty.) In this case, the
the remaining service times of all the customers in the system at time $0$ are i.i.d.\ exponentially distributed random variables with mean $1$.
We can slightly generalize this initial condition by allowing $X^n(0)$ to be larger than $n$,
but require that the waiting time of each customer in queue at time $0$ is equal to $0$.

To simplify the exposition,
we first state the stochastic-process limit theorems under the above assumption on the initial condition (see Assumption \ref{AssInitial} below).
However, we remark that we must consider much more general initial conditions in order to prove the limit theorems for the stationary distributions.
Thus, we substantially generalize Assumption \ref{AssInitial} in Section \ref{secGeneral} (see \eqref{Ass:Ia} and \eqref{Ass:Ib} there),
and prove the process limit theorems in the generalized setting.

Recall that $Z^n_1(0)$ is the number of customers in phase-$1$ service at time $0$.
Let $\ell^n_i=\{\ell^n_i(t):t\ge0\}$ be the elapsed waiting time of the $i$th customer in the queue at time $t$, $i \ge 1$,
where $\ell^n_i(t) := 0$ for $i> Q^n(t)$.
\begin{assumption}[initial condition] \label{AssInitial}
  $Z^n_1(0) = 0$ and $\sum_{i=1}^{Q^n(0)} \ell^n_i(0) = 0$ w.p.1.
\end{assumption}
In particular, the condition $Z^n_1(0)=0$ implies that the remaining service times of all the customers in service at time $0$ are exponentially distributed with mean $1$.

The following FCLT shows that, for large $n$, the $M/M_{pc}/n+M_{pc}$ system behaves much like the Erlang-C model.
We remark that the asymptotic relation between the two systems is more intricate than what the diffusion limit reveals,
as the LOF limit in Theorem \ref{thFWLLN} will show.

\begin{theorem}[Diffusion Limit] \label{ThFCLT}
Assume that \eqref{Square root} holds with $\beta \in \RR$.
If Assumption \ref{AssInitial} holds and, in addition, $\wh{X}^n(0)\Ra X_0$ in $\RR$,
then $\wh{X}^n \Ra \wh X_C$ in $D$ as $n\tinf$, for $\wh X_C$ in \eqref{HWdiff}.
\end{theorem}

It is well known that the solution to the SDE \eqref{HWdiff} has a unique steady-state distribution when $\beta>0$,
which is exponential on the positive real line, and normal on the negative real line; see Theorem 1 and Corollary 2 in \cite{halfin1981heavy}.
In particular, let $\wh X_C(\infty)$ denote a random variable with that steady-state distribution, and let
$\Phi$ denote the cumulative distribution function (cdf) of the standard normal random variable. Then
\begin{align}
  \label{Eq:XCInfty}
& P(\wh X_C(\infty)>x|\wh X_C(\infty)\ge0)=e^{-\beta x}, \\
& P(\wh X_C(\infty)\le x|\wh X_C(\infty)\le0)=\Phi(\beta+x)/\Phi(\beta),
\end{align}
where
\begin{equation}
\label{Eq:XCInfty2}
P(\wh X_C(\infty)\ge0)=[1+\sqrt{2\pi}\beta\Phi(\beta)e^{\beta^2/2}]^{-1}.
\end{equation}
On the other hand, when $\beta \le 0$, the diffusion process $\wh X_C$ is either null (when $\beta = 0$) or transient (when $\beta < 0$).
This follows easily from the fact that $\wh X_C$ is distributed like an ergodic Ornstein–Uhlenbeck process on $(-\infty, 0)$, and like
a Brownian motion on $(0,\infty)$, which is driftless in the case $\beta = 0$, and has a positive drift when $\beta < 0$.

We next characterize the limits of the stationary distributions of $\wh X^n$ for the two cases in which (i) the time-limiting behavior of $\wh X_C$ exists,
namely, when $\beta < 0$, and (ii) when $\beta > 0$.
To this end, we say that a sequence of random variables $Y^n$ converges in distribution to infinity, and write $Y^n \Ra \infty$, if
$P(Y^n > M) \ra 1$ as $n\tinf$ for any $M > 0$.
\begin{theorem} \label{ThStationary}
The following hold for the sequence $\{\wh X^n(\infty) : n \ge 1\}$ as $n\tinf$.

(i)  If $\beta>0$, then $\wh{X}^n(\infty)\Ra \wh X_C(\infty)$.

(ii) If $\beta < 0$, then $\wh{X}^n(\infty)\Ra \infty$.
\end{theorem}

Given that the diffusion limit when $\beta \le 0$ is transient, it stands to reason that an analogous result to Assertion (ii) of Theorem \ref{ThStationary}
holds when $\beta = 0$; this can be proved in the special case $\theta < \mu$.
  \begin{prop}
    \label{PropDiffusion0}
Let $\beta=0$. If $\theta<1$, then $\wh{X}^n(\infty)\Ra \infty$ as $n\tinf$.
\end{prop}

\subsection{Limit Theorems Under the LOF Scaling When $\beta \le 0$} \label{secFWLLN}

Theorem \ref{ThFCLT} shows a discrepancy between the diffusion limit and the prelimit when $\beta < 0$,
as the process $X^n$ is ergodic for all $n\ge 1$, while the diffusion limit $\wh X_C$ is transient.
Theorem \ref{ThStationary} further emphasizes this discrepancy
by showing that the weak limit of the stationary distributions $\{\wh X^n(\infty) : n \ge 1\}$ is infinite. In turn, this latter results implies
that a different spatial scaling, which must be larger than $\sqrt{n}$, is needed in order to achieve a non-trivial limit for $X^n(\infty)$.
The LOF stated below identifies the exact spatial scaling of the queue in this case to be $n^{3/4}$

We also observe that a non-trivial process-limit for $n^{-3/4}X^n$
requires a time scaling of $X^n$, because $n^{-3/4} X^n \Ra 0 \eta$ over any compact interval by Theorem \ref{ThFCLT}.
As we show below, the appropriate time scaling is $n^{1/4}$. Hence, we consider the process
\bes
\label{Eq:LLFScaling}
\wt{X}^n(t) := \frac{X^n(n^{1/4} t) - n}{ n^{3/4}}, \quad t \ge 0.
\ees
The next theorem characterize the weak limit of $\wt X^n$ as the unique solution to an initial-value problem (IVP), which is why we refer
to that limit as a fluid limit. (It is an LOF limit due to the spatial scaling, which is of lower order than the typical spatial scaling by $n$
that gives rise to functional weak laws.)
\begin{theorem}[LOF limit] \label{thFWLLN}
  \label{Thm:TransientF}
  Assume that \eqref{Square root} holds with $\beta \le 0$.
  If Assumption \ref{AssInitial} holds and in addition, $\wt{X}^n(0)\Ra x_0$ in $\RR$, where $x_0 \ge 0$ is deterministic,
then $\wt{X}^n\Ra x_F$ in $D$ as $n\tinf$, where $x_F$ is the unique solution to the IVP
  \begin{equation}
    \label{Eq:FODE}
    \dot x_F = -\beta-\frac{\theta^2}{2}x_F^2, \quad x_F(0) = x_0.
  \end{equation}
  In particular
   \begin{equation}
    \label{Eq:FXSolu}
     x_F(t)=\frac{\sqrt{-2\beta}}{\theta}\frac{(\sqrt{-2\beta}+\theta x_0)(1-e^{-\sqrt{-2\beta}\theta t})+2\theta x_0e^{-\sqrt{-2\beta}\theta t}}
    {(\sqrt{-2\beta}+\theta x_0)(1-e^{-\sqrt{-2\beta}\theta t})+2\sqrt{-2\beta}e^{-\sqrt{-2\beta}\theta t}}, \quad  \text{when $\beta < 0$},
  \end{equation}
  and
  \begin{equation}
    \label{Eq:FXSolu2}
    x_F(t)=\frac{2x_F(0)}{2+\theta^2 x_F(0) t}, \quad \text{when $\beta = 0$}.
  \end{equation}
\end{theorem}
A point $a \in \RR_+$ is a stationary point of $x_F$ if $x_F(t) = a$ for all $t \ge 0$ whenever $x_F(0) = a$;
it is globally asymptotically stable (and then also the unique stationary point) if $x_F(t) \ra a$ as $t\tinf$, for any solution $x_F$ to \eqref{Eq:FODE}.
Let
\bequ \label{x*}
x^*:=\sqrt{-2\beta}/\theta
\eeq
The following corollary follows immediately from Theorem \ref{thFWLLN}.
\begin{coro}[Stability of the IVP]
  \label{Cor:xstar}
$x^*$ is a globally asymptotically stable stationary point of \eqref{Eq:FODE}.
\end{coro}
In fact, any solution $x_F$ to \eqref{Eq:FODE} approaches $x^*$ monotonically, as can be seen from \eqref{Eq:FXSolu} and \eqref{Eq:FXSolu2},
or alternatively, from the fact that $\dot x < 0$ for all $x > x^*$ and $\dot x > 0$ for all $x < x^*$ (the latter being relevant only when $\beta < 0$).

Analogously to Theorem \ref{ThStationary}, we can prove that $x^*$ is the weak limit for the stationary random variables
$$\wt{X}^n(\infty):=\frac{X^n(\infty)-n}{ n^{3/4}}.$$
\begin{theorem}
  \label{ThStationaryF}
  If $\beta\le0$, then $\wt{X}^n(\infty)\Ra x^*$ in $\RR$ as $n\tinf$.
\end{theorem}
Let $Q^n(\infty)$ denote a random variable with the steady state distribution of the queue process; $Q^n(\infty) = (X^n(\infty) - n)^+$.
  Since $x^* > 0$ when $\beta < 0$, Theorem \ref{ThStationaryF} implies that $Q^n(\infty)$ is $\Theta_P(n^{3/4})$.

In ending, we remark that the time scaling in $\wt X^n$ implies that the {\em relaxation time} of $X^n$, namely, the time it takes it to converge to its steady state,
is increasing without bound in $n$ when $\beta \le 0$.

\subsection{Generalizing the Initial Condition} \label{secGeneral}
The process-limit results in Theorems \ref{ThFCLT} and \ref{Thm:TransientF} are both achieved under Assumption \ref{AssInitial}.
However, to prove the limit theorems for the stationary distributions, we need to allow for more general initial conditions.
(In particular, Theorems \ref{ThStationary} and \ref{ThStationaryF} will be proved by initializing the corresponding processes according to their stationary
distribution.)
To this end, we consider initial conditions in which the service time of each customer that is in service at time $0$ has two phases:
phase-$1$ (corresponding to the delay that customer experienced in queue), and phase-$2$, which is exponentially distributed with mean $1$.

Clearly, if the remaining phase-$1$ service time of sufficiently many customers that are initially in the system is sufficiently large
(for example, if we initialize all the customers with at least $c$ time units of remaining phase-$1$ service time, for some $c > 0$), then the system
will be temporarily overloaded, in the sense that its total service capacity will be smaller than the arrival rate
over an initial period. Then over that initial overload period, the order of size of the queue would be larger than the order of the spatial scalings
in Theorems \ref{ThFCLT} and \ref{Thm:TransientF}.
Hence, to generalize the assumption on the initial condition in Theorem \ref{ThFCLT} and Theorem \ref{Thm:TransientF},
we must enforce regularity conditions that prohibit such overload incidents.


Assume that the customers in phase-$1$ service are numbered according to the order in which they entered service,
and let $r^n_j(t)$ be the remaining service time of customer $j$ at time $t$, $1 \le j \le Z^n_1(t)$, for all $t \in \{s : Z^n_1(s) > 0\}$.
Recall also that $\ell^n_i(t)$ is the elapsed waiting time of the $i$th customer in the queue at time $t$.
Let
\begin{equation} \label{Ln}
  L^n(t):=\sum_{j=1}^{Z^n_1(t)}r^n_j(t)+\sum^{Q^n(t)}_{i=1}\ell^n_i(t), \quad t \ge 0.
\end{equation}

\begin{prop}
  \label{PropSL}
$L^n$ possesses a unique stationary distribution, which is also the limiting distribution of $L^n(t)$ as $t\tinf$.
\end{prop}

\begin{proof}
Similarly to $X^n$, $L^n$ is a regenerative process, regenerating when $X^n$ hits state $0$, namely, when the system empties.
It follows from Theorem \ref{ThErgodicity} that the expected cycle length of $L^n$ is finite, so that $L^n$ is positive recurrent, implying the result.
\end{proof}
We let $L^n(\infty)$ denote a random variable that has the stationary distribution of $L^n$.
\begin{prop}
  \label{PropL}
  There exists a constant $M>0$, such that, for all $n \in \ZZ_+$,

(a) $E[L^n(\infty)]\le n^{1/2} M$;

(b) if $\beta>0$, then $E[L^n(\infty)]<M$.
\end{prop}
Let
\bes
\wh L^n(0) := \frac{L^n(0)}{ n^{1/2}} \qandq \wt L^n(0) := \frac{L^n(0)}{ n^{3/4}},
\ees
and consider the families of initial conditions satisfying the following, for a random variable $X_0$.
\begin{align}
  \tag{Ia}
  \label{Ass:Ia}
  &\mbox{$(\wh{X}^n(0), \wh L^n(0))\Ra (X_0,0)$ in $\RR^2$ as $n\tinf$.}\\
  \tag{Ib}
  \label{Ass:Ib}
  &\mbox{$(\wt{X}^n(0), \wt{L}^n(0))\Ra (X_0,0)$ in $\RR^2$ as $n\tinf$, where $X_0\ge 0$ w.p.1.}
\end{align}
We can generalize Theorem \ref{ThFCLT} and Theorem \ref{Thm:TransientF} by considering initial conditions that are sufficiently ``close'' asymptotically
(under the relevant scaling) to the stationary distribution. 
\begin{thmstar}{ThFCLT}
  \label{ThFCLT2}
 Assume that \eqref{Square root} holds with $\beta \in \RR$. If \eqref{Ass:Ia} holds, then $\wh{X}^n\Ra \wh X_C$ in $\D$ as $n\tinf$.
\end{thmstar}
\begin{thmstar}{thFWLLN}
  \label{ThLLF2}
 Assume that $\beta\le0$. If \eqref{Ass:Ib} holds, then $\wt{X}^n\Ra x_F$ in $\D$ as $n\tinf$,
 where, conditional on $\{X_0 = x_0\}$, for a positive scalar $x_0$, $x_F$ is the unique solution to the IVP \eqref{Eq:FODE}.
\end{thmstar}
Note that the initial conditions in Theorems \ref{ThFCLT2} and \ref{ThLLF2} hold trivially if Assumption \ref{AssInitial} holds,
and due to Proposition \ref{PropL}, when the system is stationary, namely, when $L^n(0) \deq L^n(\infty)$.
More generally, \eqref{Ass:Ia} and \eqref{Ass:Ib} hold whenever $L^n(0)$ and $L^n(\infty)$ have the same order of magnitude,
i.e., when $L^n(0) = O_P(L^n(\infty))$.

%

\section{Sample-Path Representation} \label{secTransient}

Let $A$, $S$ and $R$ be three independent unit-rate Poisson processes.
We represent the Poisson arrival process in system $n$ via $A^n(t) := A(\lambda^n t)$, $t\ge0$,
and exploit the memoryless property of the exponential distribution to characterize the departures from service and
abandonment. In particular, for $D^n(t)$ and $R^n(t)$ denoting the number of departures from service and number of abandonment by time $t$ in system $n$,
respectively, we have
$$D^n(t) = S \left(\int^t_0Z_2^n(s)ds \right) \qandq R^n(t) = R \left(\theta\int^t_0Q^n(s)ds \right), \quad t\ge 0.$$
Then
\begin{equation}
  \label{Eq:s}
  X^n(t)=X^n(0)+A(\lambda^nt)-S\left(\int^t_0 Z^n_2(s)ds\right)-R\left(\theta\int^t_0 Q^n(s)ds\right), \quad t \ge 0.
\end{equation}
Notice that the following basic equalities hold:
\begin{equation}
  \label{Eq:ssss}
  Q^n=(X^n-n)\vee0,\quad Z^n=X^n\wedge n, \quad Z^n=Z^n_1+Z^n_2.
\end{equation}
To fully characterize $X^n$, we need to characterize $Z^n_2$, or equivalently, $Z^n_1$.
Let $Z^n_0(t)$ denote the number of customers who were in the system initially (at time $0$), and are in their phase-$1$ service at time $t$.
Let $T^n_0$ be the time in which the last customer from the initial queue leaves the queue, either by entering service or by abandoning the queue; in particular,
at any $t < T^n_0$ there are customers in queue that were waiting in the queue at time $0$, and there are no such customers in the queue at any time $t \ge T^n_0$.
For any $t\ge0$, let $w^n(t)$ be the minimum between $t$ and the waiting time of the head-of-line customer.
We set $w^n(t) := 0$ if $Q^n(t) = 0$.

Now, if a departure from service occurs at time $s\in[T^n_0,t]$ and $Q^n(s-) > 0$, then
the customer at the head of the line begins his phase-$1$ service, and that customer is still in phase $1$ at time $t$ if and only if $\theta w^n(s-)+s>t$.
Note that the latter statement holds trivially if $Q^n(s-) = 0$, because then $w^n(s-) = 0$. We can therefore characterize $Z^n_1$ via the departure process
as follows.
\begin{equation}
  \label{Eq:ss}
  Z^n_1(t) = Z^n_0(t) + \int^t_{T^n_0\wedge t}1\{\theta w^n(s-)+s>t\}dD^n(s), \quad t\ge0.
\end{equation}

To characterize the process $w^n$, we number the customers that arrive after time $0$ by the order of their arrival,
and denote by $E^n_k$ by the arrival time of the $k$th customer to system $n$,
i.e., $E^n_k := \inf\{t : A^n(t)=k\}$.
Let $T^n_k$ denote the patience time of the $k$th arrival to system $n$, so that $\{T_k^n : k \ge 1\}$
is a sequence of independent exponential random variable with mean $1/\theta$ for each $n \ge 1$.
Under the FIFO policy, the arrival time of any customer that is in queue at time $t$ is no less the arrival time of the head-of-line customer
at that time, the latter being equal to $t-w^n(t)$.
Hence, if the $k$th customer arrives during the time interval $[t-w^n(t), t)$, then that customer is still in the system (waiting in queue) at time $t$
if and only if $E^n_k+T^n_k>t$.
This gives
\begin{equation}
  \label{Eq:sss}
  Q^n(t)=\int^t_{t-w^n(t)}1\{E^n_{A^n(s)}+T^n_{A^n(s)}> t\}dA^n(s)+Q^n_0(t),\mbox{ for all } t\ge0,
\end{equation}
where $Q^n_0(t)$ is the number of customers that were waiting in queue at time $0$, and are still waiting in queue at time $t$.
Note that, due to abandonment, $Q^n_0(t) \le (Q^n(0)-D^n(t))^+$, and that there are no waiting customers at time $t$ if there are idle agents, so that
\begin{equation}
  \label{Eq:ZW}
  (Z^n-n)w^n=0\eta.
\end{equation}

If we assume that $L^n(0)=0$, so that $Q^n_0=Z^n_0=0\eta$, then \eqref{Eq:s}--\eqref{Eq:sss} characterize the system's dynamics via the primitives
$A^n$, $S^n$, and $\{T^n_k:k\in\ZZ_+\}$.
When $L^n(0)>0$, the dynamics of the $n$th system depend also on $\{\ell^n_i(0)\}$ and $\{r^n_j(0)\}$.
However, as will be proved below, the impact of these two sequences is asymptotically negligible,
in that they do not alter the diffusion limit and LOF limit under our assumed initializations in \eqref{Ass:Ia} and \eqref{Ass:Ib}.


\subsection{A Martingale Representation} \label{secMartRep}


Let $\mathcal{F}^n_0$ be the the $\sigma$-algebra generated by
\begin{equation*}
  \{X^n_0,\ell^n_i(0),r^n_j(0): 1 \le i \le Q^n(0), 1 \le j \le Z^n(0)\},
\end{equation*}
augmented by including all $\mathrm P$-null sets.
For $t \ge 0$ and $n \ge 1$, let $\mathcal{F}^n := \{\mathcal{F}^n_t : t \ge 0\}$, where $\mathcal F^n_t$
is the right-continuous $\sigma$-algebra associated to the $\sigma$-algebra generated by
\begin{align*}
  \big(\mathcal{F}^n_0,\ell^n_i(s), r^n_j(s), A^n(s), D^n(s), R^n(s):\,
 1 \le i \le Q^n(t),\, 1 \le j \le Z^n(t),\, s \in [0,t]\big).
\end{align*}

Note that the processes $X^n, Q^n, Q^n_0, Z^n, Z^n_i$, $i = 0,1,2$, and $w^n$ have sample paths in $D$ by construction.
Now, $w^n(t)=1\{Q^n(t)>0\}\ell^n_1(t)$, so that $w^n$ is $\mathcal F^n$-adapted, and it therefore follows from
\eqref{Eq:s}--\eqref{Eq:ss} and the equality
\begin{equation*}
  Z^n_0(t)+Z^n_1(t)=\sum^{Z^n(t)}_{j=1}1\{r^n_j(t) > 0\}.
\end{equation*}
that $X^n$, $Q^n$, $Z^n$, and $Z^n_i$, $i=0,1,2$, are also $\mathcal F^n$-adapted.
Finally, noting that $T^n_0=\inf\{t\ge0: w^n(t)>t\}$ shows that $T^n_0$ is an $\mathcal{F}^n$-stopping time.

Consider the processes
\begin{equation*} \label{Eq:Martingales}
\begin{split}
M^n_A(t) & := A^n(t)-\lambda^n t, \qquad 
M^n_S(t) := D^n(t)-\int^{t}_0 Z_2^n(s)ds,  \\
M^n_R(t) & := R^n(t)-\theta\int^{t}_0 Q^n(s)ds ,\qquad t \ge 0.
\end{split}
\end{equation*}
Since  $Z^n_2\le n$ and $D^n(t)\le S(nt)$, we have $E[|M^n_i(t)|]<\infty$ and $E[|M^n_i(t)|^2]<\infty$, for $i=A$ and $S$.
Therefore $M^n_A$ and $M^n_S$ are square-integrable $\mathcal{F}^n$-martingales.
Note that $R^n$ and $Q^n$ have nonnegative sample paths that are bounded pathwise by the sample paths of $A^n+X^n(0)$.
For $\tau^n_k:=k1\{|X^n(0)|<k\}$, $M^n_R(\cdot\wedge \tau^n_k)$ is a square-integrable $\mathcal{F}^n$-martingale, and since
$\tau^n_k\ra\infty$ w.p.1 as $k\tinf$, $M^n_R$ is an $\mathcal{F}^n$-local martingale.
Thus, \eqref{Eq:s} admits the following {\em martingale representation}
\begin{equation*}
  X^n(t)=X^n(0)+\lambda^nt-\int^t_0Z^n_2(s)ds-\theta\int^t_0Q^n(s)ds+M^n_A(t)-M^n_S(t)-M^n_R(t).
\end{equation*}
Next, for
\begin{equation}
  \label{Eq:U1}
  U^n_1(t):=\int^t_{T^n_0\wedge t}1\{\theta w^n(s-)+s>t\}dM_S^n(s),\quad t\ge0,
\end{equation}
we can rewrite \eqref{Eq:ss} to obtain
\begin{equation}
  \label{Eq:ssc}
  Z^n_1(t)=\int^t_{T^n_0\wedge t}1\{\theta w^n(s)+s>t\}Z_2^n(s)ds+U^n_1(t)+Z^n_0(t),
\end{equation}
so that
\begin{align}
  \label{Eq:ssic}
  \int^t_0Z^n_1(s)ds&=\int^t_0(U^n_1(s)+Z^n_0(s))ds+\int^t_0\int^s_{T^n_0\wedge s}1\{\theta w^n(u)+u>s\}Z^n_2(u)du ds\nonumber\\
  &=\int^t_0(U^n_1(s)+Z^n_0(s))ds+\int^t_{T^n_0\wedge t}(\theta w^n(u))\wedge(t-u)Z^n_2(u)du.
\end{align}
The last integral in \eqref{Eq:ssic} follows from Fubini's theorem together with the fact that
\bes
\int_a^b 1\{s < c\} ds = b \wedge c - a \wedge c, \qforq a \le b,
\ees
so that
\bes
\int_u^t 1\{\theta w^n(u)+u>s\}ds = (\theta w^n(u)+u) \wedge t - u = \theta w^n(u) \wedge (t-u).
\ees
Finally, let
\begin{equation}
  \label{Eq:F}
  F^n(s,t):=\int^s_0 1\{E^n_{A^n(u)}+T^n_{A^n(u)}>t\}dA^n(u) + \theta^{-1} \lambda^n(e^{-\theta t}-e^{-\theta(t-s)}).
\end{equation}
Then, for
\begin{equation}
  \label{Eq:U2}
  U_2^n(t):=F^n(t,t)-F^n(t-{w}^n(t),t),
\end{equation}
we can rewrite \eqref{Eq:sss} as follows
\begin{equation}
  \label{Eq:sssc}
  Q^n(t)=\theta^{-1}\lambda^n(1-e^{-\theta w^n(t)})+U^n_2(t)+Q^n_0(t),\;t\ge0.
\end{equation}

Plugging \eqref{Eq:ssic} and \eqref{Eq:sssc} in \eqref{Eq:s}, and using the equality $Z=Z_1+Z_2$, give the following modified martingale representation
\begin{align}
  \label{Eq:sc}
  X^n(t)=&X^n(0)+\lambda^n t-\int^t_0Z^n(s)ds+V^n(t)+\int^t_0\left(Z^n_0(t)-\theta Q^n_0(s)\right)ds\nonumber\\
  &+\int^t_0\left(U^n_1(s)-\theta U^n_2(s)\right)ds
  +M^n_A(t)-M^n_S(t)-M^n_R(t), \quad t \ge 0,
\end{align}
where
\begin{equation}
  \label{Eq:V}
  V^n(t):=\int^t_{T^n_0\wedge t}\left(\theta w^n(s)\right)\wedge(t-s)Z_2^n(s)ds-\int^t_0\lambda^n(1-e^{-\theta w^n(s)})ds.
\end{equation}


\section{Proofs of Main Results}
\label{secMainProof}
In this section we prove the main results in the paper, building on auxiliary results whose proofs are relegated to Section \ref{secAuxProof1}.
Let
\bequ \label{beta^n}
\beta^n:=n^{-1/2}(n-\lambda^n),
\eeq
and note that, due to the square-root staffing rule in \eqref{Square root}, $\beta^n \ra \beta$ as $n\tinf$.

\subsection{Proof of Theorem \ref{ThFCLT2}}

We consider the diffusion-scaled random variables and processes
\begin{align*}
  \wh Q^n &:= n^{-1/2}Q^n, &\wh Z^n&:= n^{-1/2}(Z^n-n) & \wh{L}^n(0) &:= n^{-1/2}L^n(0), & \wh{w} &:= n^{1/2}w^n,\\
  \wh{Z}^n_1 &:= n^{-1/2}Z_1^n, & \wh{Z}^n_2 &:= n^{-1/2}(Z_2^n-n), &\wh{Q}^n_0 &:= n^{-1/2}Q^n_0, & \wh{Z}^n_0 &:= n^{-1/2}Z^n_0.
\end{align*}
We similarly consider the diffusion-scaled processes in the martingale representation
\begin{align*}
\wh{M}^n_i:=n^{-1/2}M^n_i, ~i = A,S,R, \quad \wh{U}_1^n:=n^{-1/2}U^n_1,\quad
\wh{U}_2^n:=n^{-1/2}U^n_2,\quad  \wh{V}^n:=n^{-1/2}V^n.
\end{align*}
Using the diffusion scaling in \eqref{Eq:sc} gives
\begin{align}
  \label{Eq:XD}
  \wh{X}^n(t)&=\wh{X}^n(0) - \beta^n t - \int^t_0\wh{Z}^n(s)ds+\wh{V}^n(t)+\int^t_0\left(\wh{Z}^n_0(s)-\theta \wh{Q}^n_0(s)\right)ds\nonumber\\
  & \quad +\int^t_0\left(\wh{U}_1^n(s)-\theta\wh{U}^n_2(s)\right)ds
  +\wh{M}^n_A(t)-\wh{M}^n_S(t)-\wh{M}^n_R(t).
\end{align}

The proof of Theorem \ref{ThFCLT2} is a straight forward application of the continuous-mapping theorem, given the following key result,
whose proof appears in Section \ref{SubSec:ProofD}.
\begin{prop}
  \label{Prop:D}
Assume that \eqref{Ass:Ia} holds. Then as $n\tinf$,
  \begin{itemize}
    \item[a.] $(\int^\cdot_0\wh{Q}^n_0(s)ds, \int^\cdot_0\wh{Z}^n_0(s)ds)\Ra (0\eta,0\eta) \mbox{ in } \D^2$;
    \item[b.] $(\int^\cdot_0\wh{U}^n_1(s)ds,\int^\cdot_0\wh{U}^n_2(s)ds,\wh{V}^n)\Ra (0\eta,0\eta,0\eta)\mbox{ in }\D^3$;
    \item[c.] $(\wh{M}^n_A,\wh{M}^n_S,\wh{M}^n_R)\Ra (B_1,B_2,0\eta)\mbox{ in }\D^3$, where $B_1$ and $B_2$ are two independent standard Brownian motions.
  \end{itemize}
\end{prop}

\begin{proof}[Proof of Theorem \ref{ThFCLT2}]
Using the equality $\wh{Z}^n=\wh{X}^n\wedge0$ and \eqref{beta^n} in \eqref{Eq:XD}, we have
\begin{equation}
  \label{Eq:XCConvergence}
  \wh{X}^n(\cdot)-\wh{X}^n(0)+\beta^n \eta(\cdot)-\int^\cdot_0\wh{X}^n(s)\wedge0ds\Ra \sqrt{2}B(\cdot) \qinq \D \mbox{ as }n\tinf,
\end{equation}
where $B$ is a standard Brownian motion.

By \cite[Theorem 4.1]{pang2007martingale}, 
there exists a unique solution $x\in \D$ to the integral equation
\begin{equation} \label{contMap}
  x(t)=x(0)-\beta t-\int^t_0 x(s)\wedge 0 ds+y(t),\mbox{ for all }t\ge0,
\end{equation}
and the mapping $\phi: D \ra D$, which maps the function $y$ in \eqref{contMap} to the solution $x$, is continuous in the $J_1$ topology.
Further, if $y$ is continuous, then so is $x$.
Hence, the statement of the theorem follows from \eqref{Eq:XCConvergence} and the continuous mapping theorem, by noting that
\bes
  \wh{X}^n=\phi(\wh{X}^n(\cdot)-\wh{X}^n(0)+\beta^n \eta(\cdot)-\int^\cdot_0\wh{X}^n(s)\wedge0ds),
  \ees
  and that $\wh X_C=\phi(\sqrt{2}B)$.
\end{proof}


\subsection{Proof of Theorem \ref{ThLLF2}.}

To establish the LOF limit, we consider the scaled processes
\begin{align*}
  \wt{Q}^n(t)   &:= n^{-3/4}Q^n(n^{1/4}t),   & \wt{Z}^n(t)   & := n^{-3/4}(Z^n(n^{1/4}t)-n), \\
  \wt{Z}^n_1(t) &:= n^{-3/4}Z_1^n(n^{1/4}t), & \wt{Z}^n_2(t) & := n^{-3/4}(Z_2^n(n^{1/4}t)-n), \\
  \wt{Q}^n_0(t) &:= n^{-3/4}Q^n(n^{1/4}t),   & \wt{Z}^n_0(t) & := n^{-3/4}Z^n_0(n^{1/4}t), \\
  \wt{U}_1^n(t) &:= n^{-3/4}U^n_1(n^{1/4}t), & \wt{U}_2^n(t) & := n^{-3/4}U^n_2(n^{1/4}t),\\
  \wt{V}^n(t)   &:= n^{-3/4}V^n(n^{1/4}t),   & \wt{L}^n(t)   & := n^{-3/4}L^n(n^{1/4}t),  \\
  \wt{w}^n(t)   &:=n^{1/4}w^n(n^{1/4}t),     & \wt{T}^n_0    & := n^{-1/4}T^n_0,   
\end{align*}
and $\wt{M}^n_i(t) := n^{-3/4}M^n_i(n^{1/4}t)$, for $i=A,S,R$.
Then the corresponding scaled process in \eqref{Eq:s} is represented via
\begin{align}
  \label{Eq:sF}
  \wt{X}^n(t)&=\wt{X}^n(0)-\beta^n t-n^{1/4}\int^t_0\wt{Z}^n(s)ds+\wt{V}^n(t)+n^{1/4}\int^t_0\left(\wt{Z}^n_0(s)-\theta \wt{Q}^n_0(s)\right)ds\nonumber\\
  &+n^{1/4}\int^t_0\left(\wt{U}_1^n(s)-\theta\wt{U}^n_2(s)\right)ds
%
  +\wt{M}^n_A(t)-\wt{M}^n_S(t)-\wt{M}^n_R(t).
\end{align}

The proof of Theorem \ref{ThLLF2} builds on the following three supporting propositions, whose proofs appear in Section \ref{secAuxProof1}.
Throughout, {\em we assume that \eqref{Ass:Ib} holds.}
\begin{prop}
  \label{Prop:F}
  As $n\tinf$,
  \begin{itemize}
    \item[a.] $n^{1/4}(\int^\cdot_0\wt{Z}^n_0(s)ds, \int^\cdot_0\wt{Q}^n_0(s)ds)\Ra(0\eta,0\eta)$ in $\D^2$ and $T^n_0\Ra0$ in $\RR$.
    \item[b.] $(\wt{M}^n_A,\wt{M}^n_S,\wt{M}^n_R)\Ra (0\eta,0\eta,0\eta)$ in $\D^3$.
    \item[c.] $n^{1/4}\int^\cdot_0\wt{U}^n_1(s)ds\Ra0\eta$ in $\D$.
    \item [d.]$n^{1/4}\wt{U}^n_2\Ra 0\eta$, so that $n^{1/4}\int^\cdot_0\wt{U}^n_2(s)ds\Ra0\eta$ in $\D$.
  \end{itemize}
\end{prop}

\begin{prop}
  \label{Prop:Tightness}
  $\{\wt{Q}^n:n \ge 1\}$ is $C$-tight in $\D$.
\end{prop}

\begin{prop}
  \label{Prop:V}
  As $n\tinf$ 
  \begin{equation}  \label{Eq:VX}
    \wt{V}^n(\cdot)+\frac{\theta^2}{2}\int^\cdot_0(\wt{Q}^n(s))^2ds \Ra 0\eta \mbox{~ in }\D. 
  \end{equation}
\end{prop}


For a given $\phi\in \D_0$, we say that $(y,\psi)\in \D^2$ is a solution to the Skorohod problem if
\begin{align}
  \label{Eq:Skorohod}
  &y=\phi+\psi;\\
  &\int^t_0y(s)d\psi(s)=0,\mbox{ for all }t\ge0;\nonumber\\
  &y\ge0,\; \psi(0)=0\mbox{ and } \psi\mbox{ is non-decreasing}.\nonumber
\end{align}
It is well-known (e.g., see \cite[Theorem 6.1]{chen2013fundamentals}) that the Skorohod problem in \eqref{Eq:Skorohod} admits a unique solution $(y,\psi)$,
and that $h : D_0 \ra D^2$, mapping the input $\phi$ to that solution, namely, the map defined via
\bequ \label{hmap}
h(\phi) := (y,\psi)
\eeq
is (Lipschitz) continuous in the $J_1$ topology; see Theorems 13.4.1 and 13.5.1 in \cite{whitt2002stochastic}.
(Continuity of $h$ is proved only in the uniform topology in \cite{chen2013fundamentals}.) Further, if $\phi$ is continuous, then so is $h(\phi)$.



\begin{proof}[Proof of Theorem \ref{ThLLF2}]
  Due to Proposition \ref{Prop:Tightness}, any subsequence of $\{\wt{Q}^n : n \ge 1\}$ has a further weakly converging subsequence in $\D$.
Let $\{\wt{Q}^k : k \ge 1\}$ denote such a converging subsequence, and let $Q$ denote its weak limit.
Let $\Phi^k\in D$ and $\Phi\in C$ be defined via
\begin{align}
  \label{Eq:Phik}
  &\Phi^k(t)=\wt{X}^k(t)+k^{1/4}\int^t_0\wt{Z}^k(s)ds,\\
  \label{Eq:Phix1}
  &\Phi(t)=X_0-\beta t - \frac{\theta^2}{2}\int^t_0 Q^2(s)ds.
\end{align}
By \eqref{Eq:sF}, Propositions \ref{Prop:F} and \ref{Prop:V} and the continuous-mapping theorem, it holds that
\begin{equation*}
  \Phi^k-\wt{X}^k(0)+\beta^k\eta+\frac{\theta^2}{2}\int^\cdot_0(\wt{Q}^k(s))^2ds\Ra 0\eta \qinq D \qasq k\tinf.
\end{equation*}
The convergence $\wt{Q}^k\Ra Q$ in $\D$ and the continuous mapping theorem together give
\begin{equation*}
  (\wt{Q}^k,\Phi^k)\Ra(Q,\Phi)\mbox{ in }\D^2 \mbox{ as } k \tinf.
\end{equation*}

We need the following lemma, the proof of which appears at the end of this section.
Recall $h$ from \eqref{hmap}.
\begin{lemma}
  \label{Lem:Skorohod}
$(\Phi^k, \wt{X}^k,\wt{X}^k-\Phi^k)\Ra (\Phi, h(\Phi))$ as $k\tinf$ in $\D^3$.
\end{lemma}

Denote $(X,\Psi):=h(\Phi)$, so that $X=\Phi+\Psi$ and $X\ge0$ w.p.1.\
Since $h$ maps $C_0$ to $C^2$ and $\Phi\in C_0$, we have $(X, \Psi) \in C^2$.
$\wt{X}^k\Ra X$ and the continuous mapping theorem imply that
\begin{equation*}
  \wt{Q}^k=\wt{X}^k\vee0\Ra X\vee0, \qinq \D \qasq n\tinf,
\end{equation*}
and thus $Q=X\vee0=X$, w.p.1.\
In particular, \eqref{Eq:Phix1} simplifies to
\begin{equation}
  \label{Eq:Phix}
  \Phi(t)= X_0-\beta t - \frac{\theta^2}{2}\int^t_0 X^2(s)ds.
\end{equation}

It follows from \eqref{Eq:Skorohod} and the fact that $(X,\Psi)= h(\Phi)$ that $\Psi$ is a non-decreasing process with $\Psi(0)=0$, such that
\begin{equation*}
  X=\Phi+\Psi\qandq\int^\cdot_01\{X(s)>0\}d\Psi(s)=0\eta.
\end{equation*}
Hence, conditional on $\{X_0=x_0\}$, for $x_0\ge0$, and using \eqref{Eq:Phix}, $(y,\psi):=(X,\Psi)$ satisfies the following 
\begin{align}
  \label{Eq:GSkorohod}
  &y(t)=x_0-\beta t-\frac{\theta^2}{2}\int^t_0y^2(s)ds+\psi,\\
  &\int^t_01\{y>0\}d\psi=0,\nonumber\\
  &(y,\psi)\in C^2,\;y\ge0,\;\psi(0)=0,\;\mbox{ and }\psi\mbox{ is a non-decreasing process.}\nonumber
\end{align}
The next lemma is proved at the end of this section.
\begin{lemma}
  \label{Lem:GSkorohod}
  There exists a unique solution $(y,\psi)=(x_F, 0\eta)$ to \eqref{Eq:GSkorohod} for any input $x_0\ge0$ and $\beta\le 0$,
  where $x_F$ is the unique solution to \eqref{Eq:FODE}.
Further, the function $g:\RR_+\ra C^2$, mapping $x_0$ to $(y,\psi)$, is continuous.
\end{lemma}
It follows that, conditional on $\{X_0=x_0\}$, $X=x_F$ w.p.1, so that $(X, Q, Z) = (x_F, x_F \vee 0, 0\eta)$ w.p.1.
The uniqueness of the limit implies the stated weak convergence.
\end{proof}

\begin{proof}[Proof of Lemma \ref{Lem:Skorohod}]
  For fixed $\tau>0$ and $\epsilon>0$, and for each $k \ge 1$ such that $k^{-3/4}<\ep$, define the event
\begin{equation*}
  \Xi^k \equiv \Xi^k(\ep, \tau) := \{-\ep < \wt X^k(0), ~ \inf_{t\le \tau}\wt X^k(t)\wedge0 < -3\epsilon\}.
\end{equation*}
We first show that $\Xi^k$ is an asymptotically null event in the sense that $P(\Xi^k)\rightarrow0$ as $n\tinf$.
To this end, let $t^k_1$ be such that $\wt X^k(t^k_1)<-3\epsilon$. For
$$t^k_2 :=\sup\{t<t_1^k:\wt X^k(t)\ge-\epsilon\}$$
it holds that $\wt X^k(t^k_2-) \ge -\epsilon$.
As $X^k$ is a pure jump process with jumps of size $1$ and $-1$ w.p.1,
\begin{equation*}
  \wt X^k(t_2^k)\ge \wt X^k(t_2^k-)-k^{-3/4}>-2\epsilon.
\end{equation*}
Let $\phi:=-\beta-\theta^2(Q)^2/2\in D$ so that $\Phi(t)=\Phi(0)+\int^t_0\phi(s)ds$ for $t\ge0$.
\begin{align*}
  -\epsilon>&\wt X^k(t_1^k)-\wt X^k(t_2^k)=\Phi^k(t_1^k)-\Phi^k(t_2^k)-k^{1/4}\int^{t_1^k}_{t_2^k}\wt X^k(s)\wedge0ds\nonumber\\
  \ge&\Phi(t^k_1)-\Phi(t^k_2)-2\|\Phi^k-\Phi\|_\tau+k^{1/4}(t_1^k-t^k_2)\epsilon\nonumber\\
  \ge&-(t^k_1-t^k_2)\|\phi\|_\tau - 2\|\Phi^k-\Phi\|_\tau + k^{1/4}(t_1^k-t^k_2)\epsilon.
\end{align*}
The strict inequality above can hold if either $\|\phi\|_\tau \ge k^{1/4}$ or $\|\Phi^k-\Phi\|_\tau \ge \epsilon/2$,
implying that
\begin{equation*}
  \Xi^k\subseteq\{\|\Phi^k-\Phi\|_\tau \ge \epsilon/2\}\cup\{\|\phi\|_\tau \ge k^{1/4}\}.
\end{equation*}
As both events on the right-hand side are asymptotically null under the probability measure $P$, we conclude that $P(\Xi^k)\rightarrow0$ as $n\tinf$.

Next, $X_0\ge0$ implies that $P(\wt X^k(0)>-\ep)\ra 1$.
Together with the fact that $P(\Xi^k)\ra0$, we have
\begin{equation*}
  P(\inf_{t\le\tau}\wt X^k(t)\wedge0<-3\ep)\ra0, \qforallq\ep>0\qandq\tau>0 
\end{equation*}
and thus
\bequ \label{YnLim}
\wt X^k\wedge0\Rightarrow 0\eta \qinq \D \qasq n\tinf.
\eeq
It is easy to check that $\Phi^k-\wt X^k\wedge0\in D_0$ and that
\begin{equation*}
  \left(\wt X^k\vee0, k^{1/4}\int^\cdot_0\wt X^k(s)\wedge0ds\right)=h(\Phi^k-\wt X^k\wedge0).
\end{equation*}
Now, due to \eqref{YnLim}
\begin{equation*}
  \Phi^k-\wt X^k\wedge0\Rightarrow \Phi \qinq \D \qasq n\tinf,
\end{equation*}
and so
\begin{equation*}
  (\wt X^k\vee0,\wt X^k-\Phi^k)\Rightarrow h(\Phi)\qinq \D^2 \qasq n\tinf.
\end{equation*}
Thus
\begin{equation*}
  (\Phi^k,\wt X^k\vee0,\wt X^k-\Phi^k)\Rightarrow (\Phi,h(\Phi))\qinq \D^3 \qasq n\tinf.
\end{equation*}
Writing $\wt X^k=\wt X^k\wedge0+\wt X^k\vee0$ and employing \eqref{YnLim} gives the stated limit.
\end{proof}

\begin{proof}[Proof of Lemma \ref{Lem:GSkorohod}]
First, it follows from the standard theory of ordinary differential equation that \eqref{Eq:FODE} has a unique solution. (It is easy to check that
$x_F$ in \eqref{Eq:FXSolu} and \eqref{Eq:FXSolu2} satisfies \eqref{Eq:FODE} when $\beta < 0$ and $\beta = 0$, respectively.)
Then $(x_F, 0\eta)$ trivially satisfies \eqref{Eq:GSkorohod}, and it remains to show that it is the unique element in $C^2$ to have this property.

To this end, let $(y_1,\psi_1)\in C^2$ be a solution to \eqref{Eq:GSkorohod}.
The fact that $\psi_1\ge0$ implies that
\begin{align*}
  y_1(t)-x_F(t)&=-\frac{\theta^2}{2}\int^t_0(y_1^2(s)-x_F^2(s))ds+\psi_1(t)\\
  &\ge -\frac{\theta^2}{2}\int^t_0(y_1(s)+x_F(s))(y_1(s)-x_F(s))^+ds.
\end{align*}
and
\begin{equation*}
  (y_1(t)-x_F(t))^-\le (y_1(t)-x_F(t))^++\frac{\theta^2}{2}\int^t_0(y_1(s)+x_F(s))(y_1(s)-x_F(s))^+ds,\qforallq t\ge0.
\end{equation*}
By Gronwall's inequality, for each $t$, there is a $c_t\ge0$ such that
\begin{equation}
  \label{EqyxF}
  (y_1(t)-x_F(t))^-\le c_t (y_1(t)-x_F(t))^+.
\end{equation}
As $a^->0$ implies that $a^+=0$ for $a\in\RR$, \eqref{EqyxF} implies that $(y_1-x_F)^-=0$, so that $y_1\ge x_F$.
Therefore, if either $\beta<0$ or $x_0 > 0$, we have $y_1(t)\ge x_F(t)>0$ for all $t>0$.
By \eqref{Eq:GSkorohod}, we immediately have $\psi_1=0\eta$, and thus $y_1$ solves \eqref{Eq:FODE} and must equal $x_F$.

Next, consider the case $\beta=0$ and $x_0=0$. For $t$ such that $y_1(t)>0$, we have
\begin{equation*}
  dy_1(t)/dt=-\theta^2y_1(t)^2/2<0,\qforallq t\ge0 \mbox{ such that } y_1(t)>0.
\end{equation*}
Together with $y_1(0)=0$ and $y_1\in C$, we have $y_1=0\eta$, so that $\psi_1=-y_1=0\eta$.

Finally, it follows \eqref{Eq:FXSolu} (or \eqref{Eq:FXSolu2}) and $\psi=0\eta$ that the map $x_0\mapsto(y,\psi)$ is continuous, completing the proof of Lemma \ref{Lem:GSkorohod}.
%
\end{proof}

\subsection{Proofs of Main Results for the Stationary Distributions}
In this section we prove Theorem \ref{ThStationaryF} and Proposition \ref{PropDiffusion0}.
We omit the proof of Theorem \ref{ThStationary} since its Assertion (i) follows immediately from Theorem \ref{ThStationaryF}, and the proof
of Assertion (ii) follows similar arguments to the proof Theorem \ref{ThStationaryF}.
We will need the following two supporting propositions, whose proofs are given in Section \ref{secStationary}.
\begin{prop}
  \label{Prop:DiffusionTightness}
  If $\beta>0$, then $\{\wh{X}^n(\infty) : n \ge 1\}$ is tight in $\RR$.
\end{prop}

\begin{prop}\label{Prop:FluidTightness}
  For any $\beta\in\RR$, $\{\wt{X}^n(\infty) : n \ge 1\}$ is tight in $\RR$. Further, $\wt{X}^n(\infty)\wedge0\Ra 0$ in $\RR$ as $n\tinf$.
\end{prop}

\begin{proof}[Proof of Theorem \ref{ThStationaryF}]
For each $n \ge 1$, we consider a stationary version of the processes $X^n$ and $L^n$ by taking
\begin{equation} \label{InitialChoice}
  {X}^n(0)\deq {X}^n(\infty) \qandq ~ {L}^n(0)\deq {L}^n(\infty).
\end{equation}
Due to Proposition \ref{Prop:FluidTightness}, each subsequence of $\{\wt{X}^n(\infty) : n \ge 1\}$ has a further weakly converging subsequence;
let $\{\wt{X}^k(\infty) : k \ge 1\}$ be such a converging subsequence, and let $X_0$ be its weak limit.
Then by our choice of the initial distribution, it holds that $\wt{X}^k(0)\Ra X_0$, and
the stated convergence in Proposition \ref{Prop:FluidTightness} implies that $X_0 \ge 0$ w.p.1.
Moreover, by Proposition \ref{PropL} it holds that $\wt{L}^k(0)\Ra 0$ as $k\tinf$.

Now, conditional on the event $\{X_0=x_0\}$, for $x_0 \in \RR_+$, we have $\wt{X}^n\Ra X^0_F$ in $D$ as $n\tinf$ by virtue of Theorem \ref{ThLLF2},
where $X_F^0$ is the unique solution to the IVP \eqref{Eq:FODE} with initial condition $X^0_F(0) := X_0 = x_0$.
Moreover, the stationarity of the prelimit $\{\wt X^n : n \ge 1\}$ implies that the limit $X^0_F$ is strictly stationary as well,
so that $X^0_F(t) \deq X_0$ for all $t \ge 0$.

To show that $X_0=x^*$, w.p.1., recall that any solution to the ODE in \eqref{Eq:FODE}
converges monotonically to $x^*$ as $t \tinf$.
Hence, on the event $E_0 := \{X_0 \ne x^*\}$, it holds that
\begin{equation*}
|X^0_F(t)-x^*| < |X_0-x^*| \qforallq t > 0,
\end{equation*}
in contradiction to the stationarity of $X_0$, so that $E_0$ is a $P$-null event.
Thus, the limit of all weakly converging subsequences of $\{\wt X^n(0) : n \ge 1\}$ is $x^*$, implying that $\wt X^n(0) \Ra x^*$ as $n\tinf$.
The result follows from our choice of the initial conditions in \eqref{InitialChoice}.
\end{proof}


To prove Proposition \ref{PropDiffusion0}, we need the following comparison lemma, whose proof appears at the end of this section.
Consider two $M/M_{pc}/n+M_{pc}$ systems, denoted by $\PP_1$ and $\PP_2$, both having service rate $\mu=1$.
Let the arrival rates $\lambda_i$ in $\PP_i$, $i=1,2$, satisfy $\lambda_1\ge \lambda_2$,
the abandonment rate $\theta_1$ of $\PP_1$ satisfy $0\le\theta_1<1$, and the abandonment rate $\theta_2$ of $\PP_2$ satisfy
$$\theta_2\ge\theta_1/(1-\theta_1)\ge\theta_1.$$
Note that we allow for $\theta_1=0$, in which system $\PP_1$ reduces to an $M/M/n$ system.
(In this case, we assume that all the customers that are initially in the system
have exponentially distributed remaining service times, each with mean $1$.)

Let $X_i$ denote the number-in-system process in $\PP_i$, $i=1, 2$.
If either $\theta_i>0$ or $\lm_i<n$, Theorem \ref{ThErgodicity} implies that $X_i$ has a stationary distribution, which we denote by $X_i(\infty)$.

\begin{lemma}
  \label{LemCompTrick}
If $\theta_1>0$ or $\lm_1<n$ so that $X_1(\infty)$ and $X_2(\infty)$ exist, then $X_1(\infty)\ge_{s.t.} X_2(\infty)$.
\end{lemma}

\begin{proof}[Proof of Proposition \ref{PropDiffusion0}]
We write $\wh X_C(t;\beta)$ to make explicit the dependence of the distribution of the process $\wh X_C$ on the value of $\beta$, as well as of its stationary distribution
(when $t:=\infty$).
Fix $\epsilon>0$, and consider a sequence $\{\beta^n_\ep:n \ge 1\} \subset \RR_+$ satisfying $\beta^n_\ep \ge \beta^n$ and $\beta^n_\ep\ra\ep$ as $n\tinf$.
Let a sequence of $M/M_{pc}/n+M_{pc}$ systems be labeled by $n$, with arrival rate $\lambda^n_\ep:=n-\beta^n_\ep$, service rate $1$, and patience rate $\theta/(1-\theta)$.
Denote by $X^n_\ep(\infty)$ the stationary distribution of the number-in-system process of the $n$th system.
Lemma \ref{LemCompTrick} and the existence of $X^n_{\ep}(\infty)$ imply that $X^n_{\ep}(\infty)\le_{s.t.}X^n(\infty)$, so that, for any $M > 0$,
\begin{equation*}
  P(\wh{X}^n(\infty)>M)\ge P(\wh{X}^n_\ep(\infty)>M).
\end{equation*}
On the other hand, Theorem \ref{ThStationary} gives
\begin{equation*}
  P(\wh{X}^n_\ep (\infty)>M)\ra P(\wh X_C(\infty;\ep)>M) \qasq n\tinf,
\end{equation*}
so that
\begin{equation*}
  \liminf_{n\tinf}P(\wh{X}^n (\infty)>M)\ge P(\wh X_C(\infty;\ep)>M), \qforallq M>0.
\end{equation*}
Finally, \eqref{Eq:XCInfty}--\eqref{Eq:XCInfty2} give
\begin{equation*}
  \lim_{\ep\ra 0^+} P(\wh X_C(\infty;\ep)>M)= 1.\check{}
\end{equation*}
Therefore, $P(\wh{X}^n(\infty)>M)\ra1$ as $n\tinf$ for any $M>0$, implying the result.
\end{proof}

It remains to prove Lemma \ref{LemCompTrick}.
\begin{proof}[Proof of Lemma \ref{LemCompTrick}]
We assume that the the arrival process to $\PP_1$
is the superposition of two independent Poisson streams, with stream $1$ having rate $\lambda_2$ and stream $2$ having rate $\lambda_1-\lambda_2$.
We consider a coupling of $\PP_1$ and $\PP_2$ such that (i) both systems start empty;
(ii) stream-1 arrivals to $\PP_1$ and all arrivals of $\PP_2$ follows the same Poisson process;
and (iii) any stream-1 arrival to $\PP_1$ and the corresponding arrival of $\PP_2$ have the same service time.
Using this coupling, we will show that the sojourn time of every stream-$1$ arrival to $\PP_1$ is at least as long as that of the same arrival in $\PP_2$.
We label the stream-$1$ arrivals to $\PP_1$, that also constitute the arrivals to $\PP_2$, by $1,2,\cdots$, and denote by $S_i$ the service time of customer $i$.
We denote the coupled number-in-system processes in $\PP_1$ and $\PP_2$ by $\check X_1$ and $\check X_2$, respectively.

The proof proceeds by induction.
First, customer $1$ in $\PP_2$ enters service immediately upon arrival, so that his sojourn time is $S_1$.
The same customer in $\PP_1$ may
(i) enter service immediately, and experience the same sojourn time $S_1$;
(ii) enter service after waiting in queue, so that his sojourn time is greater than $S_1$;
or (iii) abandon the system, after waiting for $\theta_1^{-1}S_1> S_1$ units of time.
In all three scenarios, the sojourn time of customer $1$ in $\PP_1$ is at least as large as in $\PP_2$.

Take the induction hypothesis that the first $j$ customers have equal or shorter sojourn times in $\PP_2$ than in $\PP_1$.
There are three cases to consider in order to show that the same is true for the $(j+1)$st customer.

\noindent{\bf Case 1: Customer $j+1$ abandons $\PP_1$.} In this case, that customer's sojourn time in $\PP_1$ is equal to $\theta_1^{-1}S_{j+1}$.
On the other hand, the sojourn time of customer $j+1$ in $\PP_2$ is bounded from above by the service requirement $S_{j+1}$ plus the patience time $\theta_2^{-1}S_{j+1}$.
Since $(1+\theta_2^{-1})S_{j+1}\le\theta_1^{-1}S_{j+1}$, the ordering of the sojourn times for the first $j$ customers in the two systems
remains to hold for the $(j+1)$st customer.

\noindent{\bf Case 2: Customer $j+1$ is served in $\PP_1$ but abandons $\PP_2$.}
For $1 \le k \le j+1$, denote by $D^1_k$ and $D^2_k$ the time when the $k$th customer leaves system $\PP_1$ and system $\PP_2$, respectively.
By the induction hypothesis, we have $D^1_k\ge D^2_k$ for $k=1,2,\cdots,j$.
Denote by $F^1_{j+1}$ the time when customer $j+1$ enters service in $\PP_1$. 
Clearly, the first $j$ customers are not in queue at this time, namely, each of them is either in service or has left $\PP_1$
(either via abandonment or service completion). 
In particular, if customer $k \le j$ is still in system $\PP_1$, this customer must be in service.
For a system with $n$ servers, we then have
\begin{equation*}
  \sum^{j+1}_{i=1}1\{D^1_i\le F^1_{j+1}\}\le n,\qandq D^1_{j+1}=F^1_{j+1}+S_{j+1}> F^1_{j+1}
\end{equation*}
so that
\begin{equation*}
  \sum^{j}_{i=1}1\{D^1_i\le F^1_{j+1}\}\le n-1.
\end{equation*}
Using $D^1_k\ge D^2_k$ for $k=1,2,\cdots,j$, we have
\begin{equation*}
  \sum^{j}_{i=1}1\{D^2_i\le F^1_{j+1}\}\le \sum^{j}_{i=1}1\{D^1_i\le F^1_{j+1}\}\le n-1,
\end{equation*}
so that there are no more than $n-1$ of the first $j$ stream-$1$ customers in $\PP_2$ at time $F^1_{j+1}$.
As system $\PP_2$ has $n$ servers and customer $j+1$ abandons system $\PP_2$, customer $j+1$ must have abandoned system $\PP_2$ by time $F^1_{j+1}$.
Therefore, customer $j+1$ has equal or shorter sojourn time in system $\PP_2$ than system $\PP_1$.

\noindent{\bf Case 3: Customer $j+1$ is served in both systems $\PP_1$ and $\PP_2$.}
As in Case 2, there are no more than $n-1$ customers of label $1, 2,\cdots,j$ present in system $\PP_2$ at time $F^1_{j+1}$.
In this case, customer $j+1$ is served in system $\PP_2$, so this customer must have entered service by time $F^1_{j+1}$,
implying that his delay in queue in $\PP_2$ is no longer than his delay in queue in $\PP_1$.
Since the service time of this customer is the same in both coupled systems,
the ordering of his sojourn times in both systems remains as in the previous two cases.


In either of the above three cases, the ordering of the sojourn times of the stream-$1$ customers
imply that $\check X_1(t) \ge \check X_2(t)$ w.p.1, so that $X_1(t) \ge_{s.t.} X_2(t)$, for all $t \ge 0$.
Since $X_i(t) \Ra X_i(\infty)$ as $t\tinf$, for $i=1,2$, by Theorem \ref{ThErgodicity} (independently of the initial condition),
The result follows from the fact that stochastic order is maintained under weak convergence \cite[Proposition 3]{kamae1977stochastic}.
\end{proof}

\section{Proofs of Supporting Results for Process Limits} \label{secAuxProof1}

In this section we prove Propositions \ref{Prop:D}--\ref{Prop:V}. In particular, we prove Proposition \ref{Prop:F} in Section \ref{subsecPropF} and Propositions \ref{Prop:Tightness}--\ref{Prop:V} in Section \ref{subsecProofProps}.
The proof of Proposition \ref{Prop:D} appears last, in Section \ref{SubSec:ProofD}, since it requires arguments from some of the previous proofs in this section.
Proofs of supporting lemmas that are used in the proofs of the propositions are given in Section \ref{SecProofsLemmas}.

\subsection{Proof of Proposition \ref{Prop:F}.} \label{subsecPropF}
We refer to the customers that are in the system at time $0$ as the ``initial customers.''

\begin{proof}[Proof of Assertion ($a$)]
For $i=1,2,\cdots, Z^n(0)$ and $j=1,2,\cdots, Q^n(0)$, let $g^n_i(t)$ be the elapsed phase-$1$ service time of the $i$th initial customer in service,
and $h^n_j(t)$ be the elapsed phase-$1$ service time of the $j$th initial customer in the queue, at time $t$.
Then
\begin{equation*}
  \int^t_0Z^n_0(s)ds=\sum^{Z^n(0)}_{i=1}g^n_i(t)+\sum^{Q^n(0)-Q^n_0(t)}_{i=1}h^n_i(t).
\end{equation*}
Notice that $g^n_i(t)\le r^n_i(0)$, for any $i=1,2,\cdots, Z^n(0)$ and $t\ge0$ and that
  \begin{equation*}
    h^n_i(t)\le \theta \ell^n_i(0)+\theta (T^n_0\wedge t)
  \end{equation*}
for an initial customer $i$ who has left the queue by time $t$.
Hence,
\begin{align}
  \label{Eq:Z0}
  \int^t_0 Z^n_0(s)ds&\le \sum^{Z^n(0)}_{i=1}r^n_i(0)+\theta\sum^{Q^n(0)}_{j=1}\ell^n_i(0)+\theta(T^n_0\wedge t) (Q^n(0)-Q^n_0(t))\nonumber\\
  &\le (1+\theta)L^n(0)+\theta(T^n_0\wedge t) (Q^n(0)-Q^n_0(t)),\mbox{ for all } t\ge0,
\end{align}
To bound $\int^t_0 Q^n_0(s)ds$, notice that each initial customer in the queue waits for at most $T^n_0\wedge t$ during $[0,t]$, so that
\begin{equation}
  \label{Eq:Q0}
  \int^t_0 Q^n_0(t)dt\le (T^n_0\wedge t)Q^n(0),\mbox{ for all } t\ge0.
\end{equation}

Now, since $\wt{Z}^n_0$ and $\wt{Q}^n_0$ are non-negative processes, \eqref{Eq:Z0} and \eqref{Eq:Q0} give
\begin{align*}
  0 & \le n^{1/4}\int^\infty_0\wt{Z}^n_0(s)ds\le (1+\theta) \wt L^n(0)+\theta T^n_0\wt{Q}^n(0),\\
  0 & \le n^{1/4}\int^\infty_0\wt{Q}^n_0(s)ds\le T^n_0\wt{Q}^n(0).
\end{align*}
Due to \eqref{Ass:Ib}, it suffices to prove that $T^n_0\Ra0$ in $\RR$, as $n\tinf$;
in particular, we need only consider the event $\{T^n_0 > 0 \mbox{ ~ for all $n$ large enough}\}$.
Let
\bes 
T^n_1:=4n^{-1}(1+\theta)(L^n(0)+1),
\ees
and note that $T^n_1\Ra0$ in $\RR$ as $n\tinf$.
For each $n \ge 1$, define the event
$$\Upsilon^n:=\{T^n_1<T^n_0 \text{ and } Z^n_2(t_0) < n/2 \textrm{ ~for some } t_0\in[T^n_1,T^n_0]\}.$$
We will show that $P(\Upsilon^n)\ra0$ as $n\tinf$.
To this end, note that, since $Z^n_1(s)=Z^n_0(s)$ for $s\le T^n_0$ (since $Z^n_0(s)$ is the number of initial customers that are in phase-$1$ at time $s$)
and $Z^n_0+Q^n_0$ is non-increasing, it holds on the event $\Upsilon$ that, for all $s\in[0, T^n_1]$,
\begin{equation*}
  Z^n_0(s)+Q^n_0(s)\ge Z^n_0(t_0)+Q^n_0(t_0)\ge Z^n_0(t_0)=Z^n_1(t_0)> n/2.
\end{equation*}
The last inequality and \eqref{Eq:Z0} give the bounds
\begin{align}
\frac{nT^n_1}{2} < \int^{T^n_1}_0(Z^n_0(s)+Q^n_0(s))ds & \le (1+\theta)(L^n(0)+2T^n_1{Q}^n(0)) \nonumber\\
  \label{EqTn0c}
& < \frac{n T^n_1 }{4} + 2(1+\theta)T^n_1{Q}^n(0),
\end{align}
where the equality in \eqref{EqTn0c} follows from the definition of $T^n_1$.
Notice that \eqref{EqTn0c} cannot hold when $Q^n(0)\le (1+\theta)^{-1}n/8$, and so, together with \eqref{Ass:Ib},
\begin{equation*}
  P(\Upsilon^n)\le P(Q^n(0) > (1+\theta)^{-1}n/8)\ra0 \qasq n\tinf.
\end{equation*}
Thus, we need only consider sample paths on the complementary event $\Upsilon^c$. On this event, either
$T^n_0\le T^n_1$ or, if $T^n_0>T^n_1$, then there are at least $n/2$ customers in phase-$2$ service over the interval $[T^n_1, T^n_0]$,
in which case the total service rate is at least $n/2$.
In either case, for a sequence of i.i.d.\ exponentially distributed random variables $\{\mathcal{E}^n_k : k \ge 1\}$, each having rate $n/2$, it holds that
$$T^n_0 \le_{s.t.} T^n_1 + \sum_{k=1}^{Q^n(0)} \mathcal E^n_k,$$
where the latter sum is defined to be equal to $0$ on $\{Q^n(0) = 0\}$.
It follows that $T^n_0\Ra0$ in $\RR$ as $n\tinf$, implying the result.
\end{proof}
The proofs of Assertions (b) and (c) require showing that $\wt Q^n = O_P(n)$. While this follows immediately from Proposition \ref{Prop:Tightness},
we cannot use it now, since the proof of this latter proposition requires the current proof.
Thus, we next state and prove a weaker result than the tightness stated in Proposition \ref{Prop:Tightness},
whose proof appears in Section \ref{SecProofsLemmas}.
\begin{lemma}
  \label{Lem:SBQ}
  If \eqref{Ass:Ib} holds, then $\{\wt Q^n : n \ge 1\}$ is stochastically bounded.
\end{lemma}

\begin{proof}[Proof of Assertion ($b$)]
Consider the predictable quadratic variation of the (local) martingales $(\wt{M}^n_A,\wt{M}^n_S,\wt{M}^n_R)$.
As $n\tinf$, the following limits hold in $\D$
\begin{gather*}
  \langle \wt{M}^n_{A}\rangle (t) = n^{-3/2}(n^{1/4}\lambda^n) \eta \ra 0\eta,\\
  0\le\langle \wt{M}^n_{S}\rangle (t) = n^{-3/2}\int^{n^{1/4}t}_0 Z^n(s)ds\le n^{-1/4}\eta \ra 0 \eta, \\
  0\le\langle \wt{M}^n_{R}\rangle (t) = n^{-3/4}\int^{t}_0\wt{Q}^n(s)ds\Ra 0\eta,
\end{gather*}
where the last weak convergence follows from Lemma \ref{Lem:SBQ}.
Hence, $(\wt{M}^n_A, \wt{M}^n_S, \wt{M}^n_R)\Rightarrow (0\eta,0\eta,0\eta)$ in $\D^3$, as $n\rightarrow\infty$ by, e.g.,
Theorem $8.1$ in \cite{pang2007martingale}.
\end{proof}

To prove Assertions (c) and (d) we need the following lemma, whose proof appears in Section \ref{SecProofsLemmas}
\begin{lemma}
  \label{Lem:SBw}
  $\{\wt{w}^n:n\in\ZZ_+\}$ is stochastically bounded in $\D$.
\end{lemma}

%

\begin{proof}[Proof of Assertion ($c$)]
We will show that $\{n^{1/2}\int^\cdot_0 \wt{U}^n_1(s)ds : n \ge 1\}$ is stochastically bounded in $D$, from which the assertion follows immediately.
To this end, note that similar arguments to those in \eqref{Eq:ssic} give
\begin{align}
  \label{Eq:U1int}
  n^{1/2}\int^t_0\wt{U}^n_1(s)ds &= n^{1/2}\int^t_{\wt{T}^n_0\wedge t}(n^{-1/2}\theta\wt{w}^n(s_1-))\wedge(t-s_1)d\wt{M}_S^n(s_1)\nonumber\\
  &= \int^t_{\wt{T}^n_0\wedge t}\theta\wt{w}^n(s_1-)d\wt{M}_S^n(s_1) \\
  &\quad -n^{1/2}\int^t_{\wt{T}^n_0\wedge t}\left(n^{-1/2}\theta\wt{w}^n(s_1-)-t+s_1\right)^+d\wt{M}_S^n(s_1) \nonumber.
\end{align}
Since $\wt{M}^n_S$ is an $\mathcal{F}^n_t$-martingale and $\wt{w}^n(s_1-)$ is a predictable process,
$\int^\cdot_0\wt{w}^n(s_1-)d\wt{M}_S^n(s_1)$ is also an $\mathcal{F}^n_t$-martingale, with corresponding predictable quadratic variation process
\begin{align*}
  \left\langle \int^t_0\theta\wt{w}^n(s_1-)d\wt{M}_S^n(s_1) \right\rangle
  = \int^t_0(\theta\wt{w}^n(s_1-))^2d\langle \wt{M}_S^n\rangle(s_1) 
  \le (\|\theta\wt{w}^n\|_t)^2\langle \wt{M}_S^n\rangle(t).
\end{align*}
If follows from Lemma \ref{Lem:SBw} and the proof of Assertion (b) that
\begin{equation*}
  \left\langle \int^\cdot_0\theta\wt{w}^n(s_1-)d\wt{M}_S^n(s_1) \right\rangle\Ra0\eta \qinq \D \mbox{ as } n\tinf,
\end{equation*}
implying that
\begin{equation*}
  \int^\cdot_0\wt{w}^n(s_1-)d\wt{M}_S^n(s_1)\Rightarrow 0\eta \qinq \D \mbox{ as } n\tinf,
\end{equation*}
due to the martingale FCLT (e.g., Theorem $8.1$ in \cite{pang2007martingale}).
Hence, for any $t \ge 0$,
\begin{equation}
  \label{EqU1First}
  \ep^n(t) := \sup_{s\in[0,t]}\Big|\int^{s}_{\wt{T}^n\wedge s}\wt{w}^n(s_1-)d\wt{M}_S^n(s_1)\Big|\Rightarrow 0 \qinq \RR \qasq n\tinf.
\end{equation}

To treat the second integral in the right-hand side of \eqref{Eq:U1int}, we first observe that, for $s_1 \in [0,t]$,
\begin{align*}
  (n^{-1/2}\theta\wt{w}^n(s_1-)-t+s_1)^+ &\le n^{-1/2}\theta\|\wt{w}^n\|_t 1\{s_1+n^{-1/2}\theta\wt{w}^n(s_1)\ge t\} \\
  & \le n^{-1/2}\theta\|\wt{w}^n\|_t 1\{s_1\ge t- n^{-1/2}\theta\|\wt{w}^n\|_t\}, \nonumber
\end{align*}
so that
\begin{align}
  \label{Eq:SBU12}
  &\Big|n^{1/2}\int^t_{\wt{T}^n_0\wedge t}\left(n^{-1/2}\theta\wt{w}^n(s_1-)-t+s_1\right)^+d\wt{M}_S^n(s_1)\Big| 
  \le  \|\theta\wt{w}^n\|_t\int^t_{t-n^{-1/2}\theta\|\wt{w}^n\|_t}d|\wt{M}^n_S(s_1)|.
\end{align}
Furthermore (recalling that $\tilde Z^n_2$ is centered about $n$),
\begin{equation}
  \label{EqwtMS}
  \wt{M}^n_S(t)+n^{1/4}\int^t_0(n^{1/4}+\wt{Z}^n_2(s))ds=n^{-3/4}D^n(n^{1/4}t),
\end{equation}
is a non-decreasing pure-jump process and $-n^{1/4}\le \wt{Z}^n_2\le 0$, we have
\begin{align}
  \label{Eq:SBU1}
  \int^t_{s}d|\wt{M}^n_S(s_1)|\le& \int^t_s\Big(n^{-3/4}dD^n(n^{1/4}t)+|n^{1/2}+n^{1/4}\wt{Z}^n_2(s_1)|ds_1\Big)\nonumber\\
  \le &n^{-3/4}D^n(n^{1/4}t)-n^{-3/4}D^n(n^{1/4}s)+\sqrt{n}(t-s)\nonumber\\
  =&\wt{M}^n_S(t)-\wt{M}^n_S(s)+2\sqrt{n}(t-s)+n^{1/4}\int^t_s\wt{Z}^n_2(s_1)ds_1\nonumber\\
    \le&2\|\wt{M}^n_S\|_t+2\sqrt{n}(t-s),
\end{align}
for all $0\le s\le t$.
Plugging \eqref{Eq:SBU1} in \eqref{Eq:SBU12} gives
\begin{equation}
  \label{EqU1Second}
  \left|n^{1/2}\int^t_{\wt{T}^n\wedge t}\left(n^{-1/2}\theta\wt{w}^n(s_1-)-t+s_1\right)^+d\wt{M}_S^n(s_1)\right|\le2\|\theta\wt{w}^n\|_t(\|\wt{M}^n_S\|_t+\|\theta\wt{w}^n\|_t).
\end{equation}
It follows from Lemma \ref{Lem:SBw} and Assertion (b) that the right-hand side of \eqref{EqU1Second} is tight in $\RR$.
Next, plugging \eqref{EqU1First} and \eqref{EqU1Second} in \eqref{Eq:U1int} gives
\begin{equation*}
  \Big\|n^{1/2}\int^\cdot_0\wt{U}^n_1(s)ds\Big\|_t\le \ep^n(t)+2\|\theta\wt{w}^n\|_t(\|\wt{M}^n_S\|_t+\|\theta\wt{w}^n\|_t) \qforallq t \ge 0.
\end{equation*}
Therefore, $\{n^{1/2}\int^\cdot_0\wt{U}^n_1(s)ds : n \ge 1\}$ is tight in $\D$ for any $t\ge0$.
\end{proof}

\begin{proof}[Proof of Assertion ($d$)]
For $\tau > 0$, $K > 0$ and $F^n(s,t)$ in \eqref{Eq:F}, let
\bes
\mathcal M^n(\tau, K; s,t) := \sup_{\substack{s,t\in[0, n^{1/4}\tau],\\t-s\in[0,n^{-1/4}K]}}|F^n(s,t)|,
\ees
and observe that
\begin{equation*}
  n^{1/4}\|\wt{U}^n_2\|_\tau \le 2n^{-1/2} \mathcal M^n(\tau, \|\wt{w}^n\|_\tau; s,t).
\end{equation*}
Thus, the proof of the assertion will follow if we show that, for any $\ep > 0$,
\begin{equation*}
  P\big(n^{-1/2}\mathcal M^n(\tau, \|\wt{w}^n\|_\tau; s,t)>\epsilon \big)\rightarrow0 \qasq n \tinf,
\end{equation*}
which is what we prove next.

Fix $\epsilon>0$. Since $\{\wt{w}^n : n \ge 1\}$ is stochastically bounded in $\D$ by Lemma \ref{Lem:SBw},
we can find a $K := K(\ep) > 2$, such that $P(\|\wt{w}^n\|_\tau>K)<\epsilon$ for any $n$.
Notice that the value of $F^n(s,t)$ only depends on arrival times at $(s,t]$ and the patience times of those arrivals,
implying that $F^n$ is time-invariant in its two parameters, in the sense that $F^n(s,t)$ and $F^n(s+r,t+r)$ have the same law for any $r\ge0$.

Let $J^n$ be the smallest integer satisfying $J^n+1\ge n^{1/2}\tau/K$, and let
\begin{equation*}
  I^n_j:=[n^{-1/4}K(j-1),n^{-1/4}K(j+1)]\cap [0,n^{1/4}\tau],\qforq j=1,2,\cdots, J^n.
\end{equation*}
Observe that, for $j=1,2,\cdots, J^n-1$, the length of each $I^n_j$ is $2n^{-1/4}K$ and the length of $I^n_j \cap I^n_{j+1}$ is $n^{-1/4}K$.
It holds that,
for any $s, t \in [0,n^{1/4}\tau]$ for which $[s,t]\subseteq [0,n^{1/4}\tau]$ and $t-s<n^{-1/4}K$,
the interval $[s,t]$ is contained in at least one of the intervals $\{I^n_j:j=1,2,\cdots, J^n\}$. Therefore,
\begin{align}
  P(n^{-1/2} \mathcal M^n(\tau, K; s,t)>\epsilon)
  \le& \sum^{J^n}_{j=1}P(n^{-1/2}\sup_{\substack{s,t\in I^n_j,s\le t}}|F^n(s,t)|>\epsilon)\nonumber\\
  \label{Eq:Fconv}
  \le & J^nP(n^{-1/2}\sup_{s, t\in I^n_1, s\le t}|F^n(s,t)|>\epsilon),
\end{align}
where the $2$nd inequality is due to the aforementioned time-invariance property of $F^n$, which implies that all the probabilities in the sum,
except possibly the last one (which may be smaller than the rest), are equal.

We need the following lemma, whose proof appears in Section \ref{SecProofsLemmas}.
\begin{lemma} \label{Lem:FMartingales}
For each $t>0$, $\{F^n(s,t) : s \in[0, t]\}$ is a martingale, and $\{e^{\theta t}\sup_{s\in[0,t]}|F^n(s,t)|:t\ge0\}$ is a submartingale,
both with respect to their augumented natural filtration.
\end{lemma}
Employing Lemma \ref{Lem:FMartingales}, we have
\begin{align}
  \label{Eq:Fconv2}
  &P\big(n^{-1/2}\sup_{0\le s\le t\le 2n^{-1/4}K}|F^n(s,t)|>\epsilon\big)\nonumber\\
  & \le P\big(\sup_{t\in[0, 2n^{-1/4}K]}e^{\theta t}\sup_{s\in[0,t]}|F^n(s,t)|>n^{1/2}\epsilon\big)\nonumber\\
  & \le \epsilon^{-6}n^{-3}E\bigg[ \Big(e^{2\theta n^{-1/4}K}\sup_{s\in[0,2n^{-1/4}K]}|F^n(s,2n^{-1/4}K)|\Big)^6\bigg]\nonumber\\
  & \le \epsilon^{-6}e^{12\theta n^{-1/4}K}(6/5)^6E\left[ \big(n^{-1/2}F^n(2n^{-1/4}K,2n^{-1/4}K)\big)^6\right],
\end{align}
where the second and the last inequalities follow from Doob's $L^p$-maximal inequality (e.g., \cite[Theorem 1.7]{revuz2013continuous}) for $p=6$,
for the (sub)martingales in Lemma \ref{Lem:FMartingales}.

It remains to compute $E [(n^{-1/2}F^n(2n^{-1/4}K,2n^{-1/4}K))^6]$ in order to bound the right-hand side of \eqref{Eq:Fconv}.
Note that, conditional on $A^n(t)$, the vector of arrival times $(E^n_1, E^n_2,\cdots,E^n_{A^n(t)})$ is distributed as the vector of ordered statistic
of $A^n(t)$ uniform random variables on $[0,t]$.
Therefore, 
\begin{equation*}
  \int^{t}_{0}1\{E^n_{A^n(s)}+T^n_{A^n(s)}\ge t\}dA^n(s)
\end{equation*}
is, conditional on $A^n(t)$, distributed like $\sum^{A^n(t)}_{k=1}B_k(t)$, where, for each $t \ge 0$,
$\{B_k(t) : k \ge 1\}$ is a sequence of i.i.d.\ Bernoulli random variables, each distributed like
$B_t := 1\{U+T\ge t\}$, where $U$ is uniform on $[0,t]$, and $T$ is exponentially distributed with rate $\theta$ that is independent of $U$.
Thus,
$E[B_t] = (\theta t)^{-1}(1-e^{-\theta t})$, and 
\begin{equation*}
F^n(t,t) \deq \sum^{A^n(t)}_{k=1}B_k(t)-E[B_t]\lm^nt 
\end{equation*}
Let $\bar b_n$ denote $E[B_t]$ for $t=n^{-1/4}K$;
\begin{equation*}
  \bar{b}_n:=E[B_{n^{-1/4}K}]=(n^{-1/4}\theta K)^{-1}(1-e^{-n^{-1/4}\theta K}).
\end{equation*}
Let $\varphi_n$ denote the moment generating function of $n^{-1/2}F^n(n^{-1/4}K,n^{-1/4}K)$.
Using the identity $E[a^{A^n(t)}]=\exp((a-1)\lm^n t)$ for each $a>0$,
\begin{align*}
  \varphi_n(s):=&E[\exp(sn^{-1/2}F^n(n^{-1/4}K,n^{-1/4}K))]\\
  =&E\Big[\big(E[e^{n^{-1/2}sB_1}]\big)^{A^n(n^{-1/4}K)}\Big]\exp(-s\lm^nn^{-3/4}K\bar{b}_n)\\
  =&E\Big[\big(\bar{b}_ne^{n^{-1/2}s} +1-\bar{b}_n\big)^{A^n(n^{-1/4}K)}\Big]\exp(-s\lm^nn^{-3/4}K\bar{b}_n)\\
  =& \exp\Big(n^{-1/4}\lm^nK\bar{b}_n(e^{n^{-1/2}s}-1) \Big)\exp(-s\lm^nn^{-3/4}K\bar{b}_n)\\
  =&\exp(\gamma^n(e^{n^{-1/2}s}-1-n^{-1/2}s)),\qforallq s\ge0,
\end{align*}
where
\begin{equation*}
  \gamma^n:=n^{-1/4}\lm^n K\bar{b}^n=\theta^{-1}\lm^n(1-e^{-n^{-1/4}\theta K})=O(n^{3/4}).
\end{equation*}
We claim that $\varphi^{(k)}_n(0)=O(n^{-k/8})$ for all $k\in\ZZ_+$, where $\varphi^{(k)}_n(s)$ denotes the $k$th derivative of $\varphi_n$
taking value at $s$. We let
\begin{equation*}
  g_n(s):=\gamma^n\big(e^{n^{-1/2}s}-1-n^{-1/2}s\big),\qforq s\ge0,
\end{equation*}
so that $\varphi_n=\exp(g_n)$, and note that $g_n(0)=g'_n(0)=0$ and $g^{(k)}_n(0)=O(n^{3/4-k/2})$ for $k \ge 2$.

We prove this latter claim by induction. First, for $k=1$, we have
\begin{equation*}
  \varphi'_n(0)=\varphi_n(0)g'_n(0)=0=O(n^{-1/8}).
\end{equation*}
Next, take the induction hypothesis that $\varphi^{(m)}_n(0)=O(n^{-m/8})$ for all $m \le k$, and consider the $(k+1)$st derivative:
\begin{align*}
  \varphi^{(k+1)}_n(0)=&(\varphi_ng'_n)^{(k)}(0)\\
  =&\sum^{k}_{j=0}\binom{k}{j}\varphi^{(k-j)}_n(0)g^{(j+1)}_n(0)\\
  =&\sum^{k}_{j=1}\binom{k}{j}\varphi^{(k-j)}_n(0)n^{1/4-j/2}\\
  =&\sum^{k}_{j=1}O(n^{-1/8(k-j)})n^{1/4-j/2}\\
  =&O(n^{-1/8(k-1)})n^{-1/4} \\
  =&O(n^{-1/8(k+1)}).
\end{align*}
This proves our claim that $\varphi^{(k)}_n(0)=O(n^{-k/8})$ for all $k\in\ZZ_+$. In particular, taking $k=6$ gives
\begin{equation*}
  E\Big[\big(n^{-1/2}F^n(n^{-1/4}K,n^{-1/4}K)\big)^6\Big]=O(n^{-3/4}).
\end{equation*}
Using this fact in \eqref{Eq:Fconv}, and then in the upper bound in \eqref{Eq:Fconv2}, we obtain
\begin{equation*}
  P(n^{-1/2}\mathcal M^n(\tau, K; s,t)>\epsilon)=O(n^{-1/4}),\qforallq\ep>0. \qedhere
\end{equation*}

\end{proof}

\subsection{Proofs of Propositions \ref{Prop:Tightness} and \ref{Prop:V}} \label{subsecProofProps}

To prove Propositions \ref{Prop:Tightness} and \ref{Prop:V}, we need the following two lemmas.
The proof of Lemma \ref{Lem:V} is given here, since it is needed for proving Proposition \ref{Prop:D}. The proof of Lemma \ref{Lem:SBXZ}
appears in Section \ref{SecProofsLemmas}, together with the proofs of the rest of the supporting lemmas of this section.
\begin{lemma}
  \label{Lem:V}
  If \eqref{Ass:Ib} holds, then as $n\tinf$,
  \begin{gather}
    \label{Eq:VF}
    \wt{V}^n(\cdot)-\frac{1}{2}\int^\cdot_0\theta^{2}(\wt{w}^n)^2(s)ds+\int^\cdot_0\theta(\wt{Z}^n_1(s)-\wt{Z}^n_0(s))\wt{w}^n(s)ds\Ra0\eta \mbox{ in }\D.
  \end{gather}
\end{lemma}

\begin{lemma}
  \label{Lem:SBXZ}
  If \eqref{Ass:Ib} holds, then $\{\wt{X}^n : n \ge 1\}$, $\{\wt{Z}^n : n \ge 1\}$, $\{\wt{Z}^n_1-\wt{Z}^n_{0} : n \ge 1\}$,
  and $\{\wt{Z}^n_2+\wt{Z}^n_0 : n \ge 1\}$ are stochastically bounded in $\D$.
\end{lemma}

\begin{proof}[Proof of Lemma \ref{Lem:V}]
Using the definition of $V^n$ in \eqref{Eq:V}, we have
\begin{align}
  \label{Eq:Vt}
  \wt{V}^n(t)=&\int^t_{\wt{T}^n_0\wedge t} \big[( \theta\wt{w}^n(s))\wedge(n^{1/2}t-n^{1/2}s)\big](\wt{Z}^n_2(s)+n^{1/4})ds\nonumber\\
  &-\int^t_0n^{-1/2}\lambda^n(1-e^{-n^{-1/4}\theta\wt{w}^n(s)})ds\nonumber\\
  =&\int^t_0\theta\wt{w}^n(s)\wt{Z}^n_2(s)ds + \frac{1}{2}\int^t_0(\theta\wt{w}^n(s))^2ds - \int^{\wt{T}^n_0\wedge t}_0\theta\wt{w}^n(s)(\wt{Z}^n_2(s)+n^{1/4})ds\nonumber\\
    &-\int^t_{\wt{T}^n_0\wedge t}(n^{-1/2}\theta\wt{w}^n(s)-t+s)^+(n^{3/4}+n^{1/2}\wt{Z}^n_2(s))ds\nonumber\\
    &+\int^t_0\left (n^{1/4}\theta\wt{w}^n(s)-\frac{1}{2}(\theta\wt{w}^n(s))^2-n^{-1/2}\lambda^n(1-e^{-n^{-1/4}\theta\wt{w}^n(s)})\right)ds.
\end{align}
Noting that $0\le\wt{Z}^n_2+n^{1/4}\le n^{1/4}$ and that $n^{1/4}\wt{T}^n_0=T^n_0\ra0$ in $\RR$ as $n\tinf$.
By Proposition \ref{Prop:F}(a),
we have that, for all $t > 0$,
\begin{equation}
  \label{Eq:Vt1}
  \left\|\int^{\wt{T}^n_0\wedge \cdot}_0\theta\wt{w}^n(s)(\wt{Z}^n_2(s)+n^{1/4})ds\right\|_t\le n^{1/4}\wt{T}^n_0\|\wt{w}^n\|_{t}\Ra0,\qasq n \tinf.
\end{equation}
Next, using the fact that
\begin{equation*}
  (n^{-1/2}\theta\wt{w}^n(s_1)-t+s_1)^+\le n^{-1/2}\theta\|\wt{w}^n\|_t 1\{s_1\ge t- n^{-1/2}\theta\|\wt{w}^n\|_t\},\mbox{ for }s_1\in[0,t].
\end{equation*}
We have
\begin{align}
  \label{Eq:Vt3}
  0\le&\int^t_{\wt{T}^n_0\wedge t}\left(n^{3/4}+n^{1/2}\wt{Z}^n_2(s_1)\right)(n^{-1/2}\theta\wt{w}^n(s_1)-t+s_1)^+ds_1\nonumber\\
  \le& \int^t_{0}n^{1/4}\theta \|\wt{w}^n\|_t 1\{s_1\ge t-n^{-1/2}\theta\|\wt{w}^n\|_t\}ds_1\nonumber\\
  = &n^{1/4}\theta\|\wt{w}^n\|_t(t-(t-n^{-1/2}\theta\|\wt{w}^n\|_t)^+)\nonumber\\
  \le& n^{-1/4}\theta^{2}(\|\wt{w}^n\|_t)^2\Ra0\eta \mbox{ in }\D,\mbox{ as }n\tinf,
\end{align}
where the equality follows from
\begin{equation*}
  \int^b_a1\{s\ge c\}ds=b\vee c- a\vee c \qforallq a,b,c \in \RR.
\end{equation*}
Define the functions
\bequ \label{f}
\bsplit
  f_1(x) & := \left\{\begin{array}{ll}
  (e^{-x}-1+x)/x & \text{if } x \ne 0 \\
  0  &  \text{if } x = 0,
\end{array} \right. \\ \\
f_2(x) & :=  \left\{\begin{array}{ll}
(e^{-x}-1+x-\frac{1}{2}x^2)/x^2 & \text{if } x \ne 0 \\
0 & \text{if } x = 0,
\end{array}\right.
\end{split}
\eeq
and note that both $f_1$ and $f_2$ are continuous at $\RR$.
It follows from Lemma \ref{Lem:SBw} that $n^{-1/4}\wt{w}^n\Rightarrow 0\eta$ in $\D$ as $n\tinf$, and so
$f_i(n^{-1/4}\wt{w}^n)\Rightarrow 0\eta$ in $\D$ as $n\tinf$, for $i=1,2$, by virtue of the continuous mapping theorem.
Writing $\lambda^n=n-\beta^n\sqrt{n}$, we have
\begin{align}
  &n^{1/4}\theta\wt{w}^n-1/2(\theta\wt{w}^n)^2-n^{-1/2}\lambda^n(1-e^{-n^{-1/4}\theta\wt{w}^n})\nonumber\\
  \label{Eq:Vt2}
  =&(\theta\wt{w}^n)^2f_2(\theta n^{-1/4}\wt{w}^n)-n^{-1/4}\beta^n\theta\wt{w}^n(1+f_1(n^{-1/4}\theta\wt{w}^n))\Ra0\eta \mbox{ in }\D \mbox{ as }n\tinf.
\end{align}
Using the weak limits established in \eqref{Eq:Vt1}, \eqref{Eq:Vt3}, and \eqref{Eq:Vt2} in \eqref{Eq:Vt}, gives
\begin{equation}
  \label{EqVZ2}
\wt{V}^n(\cdot)-\frac{1}{2}\int^\cdot_0\theta^{2}(\wt{w}^n)^2(s)ds-\int^\cdot_0\theta\wt{Z}^n_2(s)\wt{w}^n(s)ds\Ra0\eta \qinq \D \qasq n \tinf.
\end{equation}
Now, it follows from \eqref{Eq:ZW} and $\wt{Z}^n=\wt{Z}^n_1+\wt{Z}^n_2$ that $\wt{Z}^n\wt{w}^n=0\eta$, so that
\begin{equation}  \label{EqZW1}
\wt{Z}^n_1\wt{w}^n=-\wt{Z}^n_2\wt{w}^n_2.
\end{equation}
Finally, by Proposition \ref{Prop:F}(a) and Lemma \ref{Lem:SBw},
\begin{equation*}
  \Big\|\int^\cdot_0\wt{Z}^n_0(s)\wt{w}^n(s)ds\Big\|_t\le\|\wt{w}^n\|_t\int^t_0\wt{Z}^n_0(s)ds\Ra 0\eta\qinq \D.
\end{equation*}
This, together with \eqref{EqVZ2} and \eqref{EqZW1}, gives \eqref{Eq:VF}.
\end{proof}

\begin{proof}[Proof of Proposition \ref{Prop:Tightness}]
For $x \in D$, $\tau>0$ and $\delta > 0$, consider the modulus of continuity
\begin{equation*}
  v_\tau(x,\delta):=\sup_{t-s\le\delta}\{|x(s)-x(t)| : 0 \le s < t \le \tau\}.
\end{equation*}
Given the assumed convergence of the sequence of initial conditions $\{\wt Q^n(0) : n \ge 1\}$ the statement of the proposition will follow
from \cite[Theorem 15.5]{billingsley2009convergence} once we show that
\begin{equation}
  \label{EqvprimeQ}
  \lim_{\delta\rightarrow0}\limsup_{n\rightarrow\infty}P(v_\tau(\wt{Q}^n,\delta)\ge\epsilon)=0,\qforallq \ep>0.
\end{equation}

To estimate $v_\tau(\wt{Q}^n,\delta)$, note that, due to Proposition \ref{Prop:F} and Lemma \ref{Lem:V}, we can write \eqref{Eq:sF} as follows
\begin{align}
  \label{Eq:sF3}
  \wt{X}^n(\cdot)=&\wt{X}^n(0)-\beta^n \eta-n^{1/4}\int^\cdot_0\wt{Z}^n(s)+\frac{1}{2}\int^\cdot_0\theta^{2}(\wt{w}^n)^2(s)ds\nonumber\\
  &-\int^\cdot_0\theta(\wt{Z}^n_1(s)-\wt{Z}^n_0(s))\wt{w}^n(s)ds+\varepsilon^n(t),
\end{align}
for some $\varepsilon^n \in D$ satisfying $\varepsilon^n = o_P(1)$.
Let
\begin{equation*}
  \xi^n(t):=-\beta^n t+\frac{1}{2}\int^t_0\theta^{2}(\wt{w}^n)^2(s)ds-\int^t_0\theta(\wt{Z}^n_1(s)-\wt{Z}^n_0(s))\wt{w}^n(s)ds, \quad t\ge0.
\end{equation*}
Since $\wt{Z}^n=\wt{X}^n\wedge0$, we have
\begin{align}
  \label{EqXve}
  \wt{X}^n(t)=\xi^n(t)-n^{1/4}\int^t_0\wt{X}^n(s)\wedge0ds+\varepsilon^n(t), \quad t \ge 0.
\end{align}
Fix $0\le s\le t\le \tau$.
Conditional on the event $\mathcal E^n_+ := \{\inf_{u\in[s,t)}\wt{Q}^n(u)>0\}$, we have that $\wt{X}^n(u)=\wt{Q}^n(u)>0$ for all $u\in[s,t)$,
in which case \eqref{EqXve} implies that
\begin{equation*}
  |\wt{Q}^n(t)-\wt{Q}^n(s)|\le |\xi^n(t)-\xi^n(s)|+\|\varepsilon^n\|_\tau.
\end{equation*}
Next consider the event $\mathcal E^n_0 := \{\inf_{u\in[s,t)}\wt{Q}^n(u)=0\}$.
Take
$$s_0 := \inf\{u\in[s,t):\wt{Q}^n(u)=0\} \quad \text{and} \quad t_0 := \sup\{u\in[s,t):\wt{Q}^n(u)=0\},$$
and note that $\wt{Q}^n$ is a pure jump process, so that $s_0 < t_0$ w.p.1.
Then $\wt{Q}^n(s_0)=\wt{Q}^n(t_0-)=0$, and $\wt{X}^n(u)=\wt{Q}^n(u)>0$ for all $u \in[s,s_0)\cup[t_0,t)$.
Thus, on $\mathcal E^n_0$,
\begin{align*}
  |\wt{Q}^n(t)-\wt{Q}^n(s)|\le& |\wt{Q}^n(s_0)-\wt{Q}^n(s)|+|\wt{Q}^n(t)-\wt{Q}^n(t_0-)|\\
  \le&|\xi^n(s_0)-\xi^n(s)|+|\xi^n(t)-\xi^n(t_0)|+2\|\varepsilon^n\|_\tau.
\end{align*}
Overall we see that
\begin{equation*}
  |\wt{Q}^n(t)-\wt{Q}^n(s)|\le 2\sup_{s_1,t_1\in[s,t]}|\xi^n(s_1)-\xi^n(t_1)|+2\|\varepsilon^n\|_\tau,
\end{equation*}
and thus
\begin{equation}
  \label{Eq:xi2}
 v_\tau(\wt{Q}^n,\delta)\le 2v_\tau(\xi^n,\delta)+2\|\varepsilon^n\|_\tau.
\end{equation}

Now,
\begin{equation*}
  |\xi^n(t)-\xi^n(s)|\le (t-s) \big(-\beta^n +\frac{1}{2}\theta^{2}(\|\wt{w}^n\|_\tau)^2+\theta\|\wt{Z}^n_1-\wt{Z}^n_0\|_\tau\|\wt{w}^n\|_\tau\big),
\end{equation*}
and so, Lemma \ref{Lem:SBw} and Lemma \ref{Lem:SBXZ} imply that, for any $\ep>0$, there is an $M := M(\ep) >0$ for which
\begin{equation*}
  \limsup_{n\tinf}P(|\xi^n(t)-\xi^n(s)|\ge M(t-s))\le \ep. 
\end{equation*}
Thus,
\begin{equation}
  \limsup_{n\tinf}P(v_\tau(\xi^n,\delta)\ge M\delta)\le \ep,\mbox{ for all }\delta \in [0,\tau),
\end{equation}
implying that 
\begin{equation*}
  \lim_{\delta\ra0}\limsup_{n\tinf}P(v_\tau(\xi^n,\delta)\ge \ep')=0, \qforallq \ep' > 0.
\end{equation*}
This, together with \eqref{Eq:xi2} and the fact that $\varepsilon^n=o_P(1)$, gives \eqref{EqvprimeQ}, proving the statement of the proposition.
\end{proof}

\begin{proof}[Proof of Proposition \ref{Prop:V}]
We start by proving that
\begin{equation}
  \label{EqQ0w}
  \wt{Q}^n-\wt{Q}^n_0-\wt{w}^n\Ra0\eta\qinq \D.
\end{equation}
To this end, consider the LOF-scaled version of \eqref{Eq:sssc},
\begin{align} \label{Q-Q0-w}
  \wt{Q}^n-\wt{Q}^n_0&=\theta^{-1} n^{-3/4}\lambda^n(1-e^{-\theta n^{-1/4}\wt{w}^n})+\wt{U}_2^n.
\end{align}
and the (continuous) function $f_1$ in \eqref{f}.
It follows from the proof of Lemma \ref{Lem:V} (the arguments below \eqref{Eq:Vt2}) that $f_1(n^{-1/4}\theta\wt{w}^n)\Ra0$ in $D$ as $n\tinf$,
so that
\begin{equation*}
  n^{-3/4}\lambda^n(1-e^{-n^{-1/4}\theta\wt{w}^n})=n^{-1}\lm^n\theta\wt{w}^n\left(1-f_1(n^{-1/4}\theta\wt{w}^n)\right)=n^{-1}\lm^n\theta\wt{w}^n+o_P(1).
\end{equation*}
Using the latter equality, $\lm^n/n \ra 1$, and Proposition \ref{Prop:F}(c) in \eqref{Q-Q0-w}, gives \eqref{EqQ0w}.

We next prove that
\begin{gather}
  \label{Eq:Z1w}
  \int^\cdot_0\left|\wt{Z}^n_1(s)-\wt{Z}^n_0(s)-\theta\wt{w}^n(s)\right|ds\Ra0\eta\mbox{ in }\D,\mbox{ as }n\tinf.
\end{gather}
Consider the LOF-scaled version of \eqref{Eq:ss};
\begin{equation}
  \label{Eq:ssF1}
  \wt{Z}^n_1(t)-\wt{Z}^n_0(t)=n^{-3/4}\int^t_{\wt{T}^n_0\wedge t}1\{n^{-1/2}\theta \wt{w}^n(s-)+s>t\}dD^n(n^{1/4}s), \quad t\ge0.
\end{equation}
Fix a constant $\tau>0$ and let
\begin{equation*}
  \Delta^n:=\sup_{\substack{t\in[0,\tau], s\in[\wt{T}^n_0\wedge t,t],\\t-s\le n^{-1/2}\theta\|\wt{w}^n\|_\tau}}|\wt{w}^n(s-)-\wt{w}^n(t)|
\end{equation*}
Using \eqref{EqQ0w}, the fact that $\wt{Q}^n_0(t)=0$ for all $t\ge \wt{T}^n_0$, and noting that the jumps of $\wt{Q}^n$ are of size $\pm n^{-1/4}$ w.p.1,
so that $\sup_{s\in[0,\tau]}|\wt{Q}^n(s)-\wt{Q}^n(s-)| \ra 0$ as $n\tinf$ w.p.1,
\begin{equation*}
  \Delta^n=\sup_{\substack{t\in[0,\tau], s\in[\wt{T}^n_0\wedge t,t],\\t-s\le n^{-1/2}\theta\|\wt{w}^n\|_\tau}}|\wt{Q}^n(s-)-\wt{Q}^n(t)|+\delta^n
  \le \sup_{\substack{0 \le s < t \le \tau,\\t-s\le n^{-1/2}\theta\|\wt{w}^n\|_\tau}}|\wt{Q}^n(t)-\wt{Q}^n(s)|+\delta^n,
\end{equation*}
where $\delta^n \Ra 0$ in $\RR$ as $n\tinf$.
Thus, Lemma \ref{Lem:SBw} and the $C$-tightness of $\{\wt{Q}^n : n \ge 1\}$ in Proposition \ref{Prop:Tightness} imply that $\Delta ^n\Rightarrow0$ in $\RR$, as $n\tinf$.
Then, for $s\in[\wt{T}^n_0\wedge t,t]$,
\begin{equation*}
  1\{n^{-1/2}\theta(\wt{w}^n(t)-\Delta^n)+s>t\}\le1\{n^{-1/2}\theta\wt{w}^n(s-)+s>t\}\le1\{n^{-1/2}\theta(\wt{w}^n(t)+\Delta^n)+s>t\}.
\end{equation*}
For $\wt T^n_\Delta:=\wt{T}^n_0+n^{-1/2}\theta(\|\wt{w}^n\|_\tau+\Delta^n)$ and $t \in [0,\tau]$, let
$\Upsilon^n_t := \{T^n_\Delta < t\}$, and note that, $T^n_0 \Ra0$ in $\RR$ as $n\tinf$ by Proposition \ref{Prop:F}(a), $w^n = O_P(1)$ by Lemma \ref{Lem:SBw},
and $\Delta^n \Ra0$ as shown above, imply together that, for all $t \in (0,\tau]$,
$$\wt T^n_\Delta \Ra 0 \qinq \RR \qasq n\tinf, \quad \text{so that} \quad P(\Upsilon^n_t) \ra 1 \qasq n\tinf.$$
Now, on the event $\Upsilon_t$,
\begin{equation*}
  t-n^{-1/2}\theta(\wt{w}^n(t)\pm\Delta^n)\ge \wt{T}^n_0=\wt{T}^n_0\wedge t,
\end{equation*}
and it follows from \eqref{Eq:ssF1} and the equality
\begin{equation*}
  \int^a_b1\{s>c\}dF(s)=F(a\vee c)-F(b\vee c),
\end{equation*}
that
\begin{align}
  \label{EqZ1w1}
  \wt{Z}^n_1(t)-\wt{Z}^n_0(t)\ge& n^{-3/4}\int^t_{\wt{T}^n_0\wedge t}1\{s>t-n^{-1/2}\theta(\wt{w}^n(t)-\Delta^n)\}dD^n(n^{1/4}t)\nonumber\\
=&n^{-3/4}D^n(n^{1/4}t)-n^{-3/4}D^n\big(n^{1/4}t-n^{-1/4}\theta(\wt{w}^n(t)-\Delta^n)\big).
\end{align}
Similarly,
\begin{align}
  \label{EqZ1w2}
  \wt{Z}^n_1(t)-\wt{Z}^n_0(t)\le n^{-3/4}D^n(n^{1/4}t)-n^{-3/4}D^n\big(n^{1/4}t-n^{-1/4}\theta(\wt{w}^n(t)+\Delta^n)\big).
\end{align}
For any $0\le s_1\le t_1\le\tau$, \eqref{EqwtMS} gives
\begin{align*}
  &n^{-3/4}|D^n(n^{1/4}t_1)-D^n(n^{1/4}s_1)-n^{5/4}(t_1-s_1)|\\
  \quad &= \Big|\int^{t_1}_{s_1}n^{1/4}\wt{Z}^n_2(s)ds+\wt{M}^n_S(t_1)-\wt{M}^n_S(s_1)\Big|\\
  &\le n^{1/4}\int^{t_1}_{s_1}|\wt{Z}^n_2(s)|ds+2\|\wt{M}^n_S\|_\tau\nonumber\\
  &\le n^{1/4}\int^\tau_0\wt{Z}^n_0(s)ds+n^{1/4}\int^{t_1}_{s_1}|\wt{Z}^n_2(s)+\wt{Z}^n_0(s)|ds+2\|\wt{M}^n_S\|_\tau\nonumber\\
  &\le n^{1/4}\|\wt{Z}^n_2+\wt{Z}^n_0\|_\tau(t_1-s_1)+n^{1/4}\int^\tau_0\wt{Z}^n_0(s)ds+2\|\wt{M}^n_S\|_\tau.
\end{align*}
Plugging $t_1=t$, and the values $t-n^{-1/2}\theta(\wt{w}^n(t)+\Delta^n)$, as well as $t-n^{-1/2}\theta(\wt{w}^n(t)-\Delta^n)$ instead of $s_1$, shows that,
for all $t \in [0, \tau]$,
\begin{align} \label{bdd1}
  &\Big|n^{-3/4}D^n(n^{1/4}t)-n^{-3/4}D^n\big(n^{1/4}t-n^{-1/4}\theta(\wt{w}^n(t)\pm\Delta^n)\big)-\theta \wt{w}^n(t)\Big| \nonumber \\
  \quad &\le \theta\Delta^n+n^{-1/4}\theta\|\wt{Z}^n_2+\wt{Z}^n_0\|_\tau(\|\wt{w}^n\|_\tau+\Delta^n)+n^{1/4}\theta\int^\tau_0\wt{Z}^n_0(s)ds+2\|\wt{M}^n_S\|_\tau \nonumber \\
  & =: \delta^n_\tau.
\end{align}
It follows from Assertions (a) and (b) of Proposition \ref{Prop:F}, Lemma \ref{Lem:SBXZ},
and the fact that $\Delta^n\Ra0$ in $\RR$, that $\delta^n_\tau\Ra0$ in $\RR$. Further, by \eqref{EqZ1w1} and \eqref{EqZ1w2},
\begin{equation*}
|\wt{Z}^n_1(t)-\wt{Z}^n_0(t)-\theta\wt{w}^n(t)|\le \delta^n_\tau \qforallq t \in [0, \tau],
\end{equation*}
so that
\begin{equation}
  \label{Eq:Z1wDelta}
  \int^\tau_{T^n_\Delta\wedge \tau}|\wt{Z}^n_1(s)-\wt{Z}^n_0(s)-\theta\wt{w}^n(s)|ds\Ra0 \qinq \RR.
\end{equation}
Finally notice that
\begin{equation*}
  \int^{T^n_\Delta}_0|\wt{Z}^n_1(s)-\wt{Z}^n_0(s)-\theta\wt{w}^n(s)|ds\le T^n_\Delta (\|\wt{Z}^n_1-\wt{Z}^n_0\|_{T^n_\Delta}+\|\theta\wt{w}^n\|_{T^n_\Delta}) \Ra0 \qinq \RR,
\end{equation*}
where the equality (order of magnitude) follows from
the stochastic boundedness of $\{\wt{Z}^n_1-\wt{Z}^n_0 : n \ge 1\}$ and $\{\wt{w}^n : n \ge 1\}$ in $D$,
established in Lemmas \ref{Lem:SBXZ} and \ref{Lem:SBw}, respectively.
Together with \eqref{Eq:Z1wDelta}, this shows that
\begin{equation*}
  \int^\tau_0|\wt{Z}^n_1(s)-\wt{Z}^n_0(s)-\theta\wt{w}^n(s)|ds\Ra0 \qinq \RR, \qforallq \tau > 0.
\end{equation*}
The uniform convergence over compact intervals in \eqref{Eq:Z1w} follows from to the monotonicity in $\tau$ of the integral; see \cite[Lemma 4.1]{dai1995positive}.

Now,
\begin{gather}
  \label{Eq:Z01wint0}
  \left|\int^t_0\wt{w}^n(s)(\theta\wt{w}^n(s)-\wt{Z}^n_1(s)+\wt{Z}^n_0(s))ds\right|\le\|\wt{w}^n\|_t\int^t_0\left|\theta\wt{w}^n(s)+\wt{Z}^n_0(s)-\wt{Z}^n_1(s)\right|ds,
\end{gather}
for all $t\ge0$.
It follows from \eqref{Eq:Z1w} and the fact that $\wt{w}^n =O_P(1)$,
that the right-hand side of \eqref{Eq:Z01wint0} is stochastically bounded in $\RR$ for each $t \ge 0$, and since it is also non-decreasing in $t$,
\begin{equation*}
  \int^\cdot_0\wt{w}^n(s)(\theta\wt{w}^n(s)-\wt{Z}^n_1(s)+\wt{Z}^n_0(s))ds=o_P(1),
\end{equation*}
so that
\begin{equation}
  \label{Eq:Z10wint}
  \int^\cdot_0 (\wt{Z}^n_1(s)-\wt{Z}^n_0(s))\wt{w}^n(s)ds=\int^\cdot_0\theta(\wt{w}^n(s))^2ds+o_P(1).
\end{equation}
On the other hand, for all $t \ge 0$,
\begin{align}
  \label{Eq:w2int}
  &\left|\int^t_0 \Big( (\wt{w}^n(s))^2-(\wt{Q}^n(s))^2\Big)ds\right|\nonumber\\
  & \quad =\left|\int^t_0 (\wt{w}^n(s)+\wt{Q}^n(s))(\wt{w}^n(s)-\wt{Q}^n(s))ds\right|\nonumber\\
  & \quad \le(\|\wt{w}^n\|_t+\|\wt{Q}^n\|_t)\int^t_0\left(\left|\wt{w}^n(s)+\wt{Q}^n_0(s)-\wt{Q}^n(s)\right|+\wt{Q}^n_0(s)\right)ds.
\end{align}
By Proposition \ref{Prop:F}(a), \eqref{EqQ0w}, and the facts that $\wt{w}^n = O_P(1)$ and $\wt{Q}^n = O_P(1)$,
the right-hand side of \eqref{Eq:w2int} weakly converges to $0$ in $\RR$ as $n\tinf$, for any $t\ge0$.
Notice that the right-hand side of \eqref{Eq:w2int} is non-decreasing in $t$, we obtain
\begin{gather*}
  \int^\cdot_0 \Big((\wt{w}^n(s))^2-(\wt{Q}^n(s))^2\Big)ds \Ra 0 \eta \qinq D \qasq n \tinf,
\end{gather*}
so that
\begin{equation}
  \label{Eq:w2int2}
  \int^\cdot_0(\wt{w}^n(s))^2ds=\int^\cdot_0(\wt{Q}^n(s))^2ds+o_P(1).
\end{equation}
The statement of the proposition follows by employing \eqref{Eq:w2int2} in  \eqref{Eq:Z10wint}, and then in \eqref{Eq:VF}.
\end{proof}

\subsection{Proof of Proposition \ref{Prop:D}.}\label{SubSec:ProofD}
We now prove Proposition \ref{Prop:D}, building on some of the previous arguments.
Of course, condition \eqref{Ass:Ia} is stronger than condition \eqref{Ass:Ib}, and we can therefore use Propositions \ref{Prop:F}--\ref{Prop:V}
in the current proof.

\begin{proof}[Proof of Assertion ($a$)]
The inequalities in \eqref{Eq:Z0} and \eqref{Eq:Q0} give
\begin{gather}
  \label{EqZ0h}
  \int^\infty_0\wh{Z}^n_0(s)ds\le \wh L^n(0)+T^n_0\wh{Q}^n(0) \qandq
  \int^\infty_0\wh{Q}^n_0(s)ds\le T^n_0\wh{Q}^n(0).
\end{gather}
The weak limit $T^n_0\Ra0$ in $\RR$ as $n\tinf$ in Proposition \ref{Prop:F}(a) implies the assertion.
\end{proof}

\begin{proof}[Proof of Assertion ($b$)]
Notice that
\begin{gather*}
  \wh{U}^n_1(t)=n^{1/4}\wt{U}^n_1(n^{-1/4}t),\;\;\wh{U}^n_2(t)=n^{1/4}\wt{U}^n_2(n^{-1/4}t)\mbox{ and }\wh{V}^n(t)=n^{1/4}\wt{V}^n(n^{-1/4}t), \quad t \ge 0.
\end{gather*}
Proposition \ref{Prop:F}(d) implies that $\wh{U}^n_2\Ra0\eta$ in $D$, and thus $\int^\cdot_0\wh{U}^n_2(s)ds\Ra0\eta$ in $\D$, as $n\tinf$.

To prove
\begin{equation}
  \label{EqU1h}
  \int^\cdot_0\wh{U}^n_1(s)ds= n^{1/2}\int^{n^{-1/4}\cdot}_0\wt{U}^n_1(s)ds\Ra0\eta\qinq \D \qasq n\tinf.
\end{equation}
Using similar arguments as in the proof of Proposition \ref{Prop:F}(a), one can check that, under \eqref{Ass:Ia}, $n^{1/4}T^n_0\Ra0$ in $\RR$ as $n\tinf$.
Inspecting the proof of Proposition \ref{Prop:F}(c) (see, in particular, \eqref{Eq:U1int}, \eqref{EqU1First}, and \eqref{EqU1Second}),
it is sufficient to prove that $\|\wt{w}^n\|_{n^{-1/4}\tau}\Ra0$ in $\RR$ for all $\tau\ge0$.
Notice that $\wt{w}^n(\wt{T}^n_0)\le n^{1/4}T^n_0$ and $\wt{Q}^n_0(s)=0$ for $s\ge\wt{T}^n_0$.
Then for $\tau\ge0$,
\begin{align} \label{w bd}
  \|\wt{w}^n\|_{n^{-1/4}\tau} & \le n^{1/4}{T}^n_0+\sup\{\wt{w}^n(s):s\in[\wt{T}^n_0\wedge(n^{-1/4}\tau),n^{-1/4}\tau]\} \nonumber \\
  & \le n^{1/4}T^n_0+\|\wt{Q}^n\|_{n^{-1/4}\tau}+\sup\{|\wt{w}^n(s)-\wt{Q}^n(s)-\wt{Q}^n_0(s)|:s\in[\wt{T}^n_0,n^{-1/4}\tau]\} \nonumber \\
  & \le n^{1/4}{T}^n_0+\|\wt{Q}^n\|_{n^{-1/4}\tau}+\|\wt{w}^n-\wt{Q}^n-\wt{Q}^n_0\|_{n^{-1/4}\tau}.
\end{align}
Now,
  \begin{equation*}
  \bsplit
    \|\wt{Q}^n\|_{n^{-1/4}\tau} & \le \|\wt{Q}^n(t)-\wt{Q}^n(0)\|_{n^{-1/4}\tau} + \|\wt Q^n(0)\|_{n^{-1/4}\tau} \\
    & \le \sup_{s, t\in[0,n^{-1/4}\tau]}|\wt{Q}^n(t)-\wt{Q}^n(s)|  + \|\wt Q^n(0)\|_{n^{-1/4}\tau} \Ra 0 \qinq \RR \qasq n \tinf,
    \end{split}
\end{equation*}
where the convergence follows from Proposition \ref{Prop:Tightness} and \eqref{Ass:Ia}.
Further, $\|\wt{w}^n-\wt{Q}^n-\wt{Q}^n_0\|_{n^{-1/4}\tau}\Ra0$ in $\RR$ as $n\tinf$ by \eqref{EqQ0w}.
Since $n^{1/4}T^n_0\Ra0$ in $\RR$, as was mentioned above, 
$\|\wt{w}^n\|_{n^{1/4}\tau}\Ra0$ in $\RR$ as $n\tinf$, for $\tau>0$ by \eqref{w bd}.

The proof that $\wh{V}^n\Ra 0\eta$ in $D$ builds on arguments in the proof of Lemma \ref{Lem:V}, by replacing $t$ in the proof of that lemma with $n^{-1/4}t$.
Since $n^{1/2}\wt{T}^n_0=n^{1/4}T^n_0\Ra0$ and $\|\wt{w}^n\|_{n^{-1/4}\tau}\Ra 0$ in $\RR$ for all $\tau\ge0$,
the left-hand side of \eqref{Eq:Vt1}, \eqref{Eq:Vt3}, and \eqref{Eq:Vt2}, regarded as processes of $t$, are all $o_P(n^{-1/4})$.
Using this in \eqref{Eq:Vt}, gives that
\begin{equation*}
  \wt{V}^n(n^{-1/4}\cdot)-\frac{1}{2}\int^{n^{-1/4}\cdot}_0\theta^{2}(\wt{w}^n)^2(s)ds+\int^{n^{-1/4}\cdot}_0\theta(\wt{Z}^n_1(s)-\wt{Z}^n_0(s))\wt{w}^n(s)ds=o_P(n^{-1/4}).
\end{equation*}
The stochastic boundedness of $\{\wt{Z}^n_1-\wt{Z}^n_0 : n \ge 1\}$ (Lemma \ref{Lem:SBXZ}), and the fact that $\|\wt{w}\|_{n^{-1/4}\tau}\Ra0$ in $\RR$ as $n\tinf$,
imply that
\begin{equation*}
  \wh{V}^n=n^{1/4}\wt{V}^n(n^{-1/4}\cdot)\Ra0\eta,\qinq D\qasq n\tinf. \qedhere
\end{equation*}
\end{proof}

\begin{proof}[Proof of Assertion ($c$)]
By the Poisson FCLT (e.g., Theorem 4.2 in \cite{pang2007martingale}),
\begin{equation}
  \label{Eq:PFCLT}
  (\frac{A(nt)-nt}{\sqrt{n}}, \frac{S(nt)-nt}{\sqrt{n}},  \frac{S(nt)-nt}{\sqrt{n}})\Ra (B_1,B_2, 0\eta), \qinq D^3 \qasq n\tinf,
\end{equation}
for two independent standard Brownian motions $(B_1,B_2)$.
Notice that $\wh{M}^n_A$, $\wh{M}^n_S$, and $\wh{M}^n_S$ are the compositions of the scaled compensated Poisson processes in \eqref{Eq:PFCLT} with the time changes
\begin{equation*}
  \Phi^n_A: t\mapsto n^{-1}\lambda^nt,\quad \Phi^n_S:t\mapsto n^{-1}\int^t_0Z_2^n(s)ds, \qandq \Phi^n_R:t\mapsto n^{-1}\int^t_0Q^n(s)ds,
\end{equation*}
respectively.
By \eqref{beta^n}, \eqref{EqZ0h}, and the stochastic boundedness of $\{\wt{Q}^n : n \ge 1\}$ and $\{\wt{Z}^n_2+\wt{Z}^n_0 : n \ge 1\}$ in $\D$,
established in Lemmas \ref{Lem:SBQ} and \ref{Lem:SBXZ}, respectively,
\begin{align*}
  &n^{-1}\lambda^nt=t+o(1),\quad n^{-1}\int^\cdot_0Q^n(s)ds=\int^{n^{-1/4}\cdot}_0\wt{Q}(s)ds=o_P(1),\\
  &n^{-1}\int^\cdot_0Z^n_2(s)ds=\eta+\int^{n^{-1/4}\cdot}_0(\wt{Z}^n_2(s)+\wt{Z}^n_0(s))ds-n^{-1/2}\int^\cdot_0\wh{Z}^n_0(s)ds=\eta+o_P(1),
\end{align*}
implying that
\begin{equation*}
  (\Phi^n_A,\Phi^n_S,\Phi^n_R)\Ra(\eta,\eta,0\eta), \qinq D^3 \qasq n\tinf,
\end{equation*}
and the joint convergence in Assertion (c) follows from the continuity of the composition map, e.g., Theorem 13.2.1 in \cite{whitt2002stochastic}.
\end{proof}

\section{Remaining Proofs Regarding the Stationary Limits}
\label{secStationary}
In this section we prove Propositions \ref{Prop:DiffusionTightness}, \ref{Prop:FluidTightness},
and Proposition \ref{PropL}.


\subsection{Proof of Proposition \ref{Prop:DiffusionTightness}}  \label{Sec:Coupling}
An essential step in the proofs of Propositions \ref{Prop:DiffusionTightness} and \ref{Prop:FluidTightness} is the following stochastic-order lower bound for $X^n$.
For $n\ge 1$, consider an $M/M/n+M$ (Erlang-A) system, having independent service and patience times,
with arrival rate $\lambda^n$, service rate $1$, and patience rate $\theta$.
Let $X^n_A,Q^n_A,Z^n_A$ denote the queueing processes in this Erlang-A system, analogously to the corresponding processes $X_n$, $Q^n$, $L^n$ and $Z^n$
in the $M/M_{pc}/n+M_{pc}$.
\begin{lemma}
  \label{LemCouplingL}
$X^n_A(\infty)\le_{s.t.}X^n(\infty)$. 
\end{lemma}

\begin{proof} 
We prove the lemma by coupling $M/M_{pc}/n+M_{pc}$ and the above Erlang-A system, and showing that the inequality in the statement holds w.p.1 for the coupled systems.
In particular, we give the $M/M_{pc}/n+M_{pc}$ system and the $M/M/n+M$ system the same arrival stream
and initial condition. Let $E^n_k$ denote by the arrival epoch of the $k$th customer to the $n$th systems.
Exploiting the PASTA (Poisson arrivals See Time Averages) property, and using induction,
it is sufficient to prove that, if $X^n(E^n_k)\ge X^n_A(E^n_k)$, then $X^n(E^n_{k+1})\ge X^n_A(E^n_{k+1})$, for all $k \ge 1$, where the inequalities
hold w.p.1 for the coupled systems, from which the stochastic ordering in the statement follows.

Hence, we initialize both systems with with the same number of customers, so that $X^n(0) = X^n_A(0)$, and take the induction hypothesis that
$X^n(E^n_k) \ge X^n_A(E^n_k)$.
Consider the dynamics of the $M/M_{pc}/n+M_{pc}$ when all arrivals are ``turned off'' after the $k$th arrival, and let $(X',Q',Z'_1,\{\ell'\}, \{r'\})$ denote
the corresponding Markov process.
Let $X'_A$ be the corresponding pure-death process for the Erlang-A system with arrivals turned off after the $k$th arrival. Note that the death rate of this process
at state $m \ge 1$ is
\begin{equation*}
  d_A(m):=\theta(m-n)\vee 0 + m\wedge n,
\end{equation*}
and that $X'(E^n_k+\cdot)$ is a pure jump process with $X'(E^n_k)$ jumps until it reaches state $0$.
For $j = 1,\dots, X'(E^n_k)$, let $N_j$ denote the $j$th jump time of $X'(E^n_k + \cdot)$,
so that $N_j$ is the $j$th customer that leaves the system after $E_k$.
Due to the memoryless property of the exponential distribution, at $t\ge E^n_k$,
\begin{itemize}
  \item [(i)] The number of customers in queue is $(X'(t)-n)\vee0$, each having a remaining patience time that is exponentially distributed with rate $\theta$,
  and is independent of everything else.
  \item [(ii)] The number of customers in phase-$2$ service is $X'(t)\wedge n- \sum^{Z'_1(t)}_{i=1}1\{t\le E^n_k+r'_i(E^n_k)\}$, with each
   of those customers having a remaining service time that is exponentially distributed with rate $1$, independently of everything else.
\end{itemize}
Then $N_{j+1}-N_j$ is, conditional on $(X'(N_j), Z'_1(N_j))$, distributed as the interarrival time in a non-homogeneous Poisson process with intensity function
\begin{equation*}
d_j(t):=\theta (X'(N_j)-n)\vee0+ X'(N_j)\wedge n- \sum^{Z'_1(N_j)}_{i=1}1\{t\le N_j+r'_i(N_j)\}.
\end{equation*}
Clearly, $d_j(t)\le d_A(X'(N_j))$ for all $t\ge0$ and $j=1,2,\cdots X'(E^n_k)$, implying that,
for $j\le X'_A(E^n_k)$, the sojourn time of the process $X'_A(E^n_k+\cdot)$ in state $j$ is dominated by the corresponding sojourn time of $X'(E^n_k+\cdot)$.
Using the induction hypothesis $X'_A(E^n_k) \le X'(E^n_k)$, we conclude that $X'_A(E^n_k+t)\le X'(E^n_k+t)$ for all $t\ge0$.
Finally, since $E^n_{k+1}-E^n_k$ is independent of $(X'_A(E^n_k+\cdot), X'(E^n_k+\cdot))$,
we have that $X'_A(E^n_{k+1}) \le X'(E^n_{k+1})$, implying that $X_A(E^n_{k+1})\le X(E^n_{k+1})$ for the two coupled systems.
\end{proof}

\begin{proof}[Proof of Proposition \ref{Prop:DiffusionTightness}]
Consider a sequence of $M/M/n$ (Erlang-C) systems, each with service rate $1$, and with arrival rate $\lm^n$ to the $n$th system.
Denote by $(X^n_U,Q^n_U,Z^n_U)$ the number-in-system process, the queue-length process, and the number-in-service process in the $n$th Erlang-C system.
Notice that the $M/M/n$ system can be regarded as an $M/M_{pc}/n+M_{pc}$ system with no abandonment,
so that we can apply the coupling in Lemma \ref{LemCompTrick} between the two $M/M_{pc}/n+M_{pc}$ systems (one with abandonment rate that is equal to $0$,
and the other with rate $\theta$).
Then $\beta^n\ra\beta>0$ as $n\tinf$ implies that there exists $N$, such that $\beta^n>0$ for $n\ge N$,
so that $X^n_U(\infty)$ exists and Lemma \ref{LemCompTrick} implies that $X^n_U(\infty)\ge_{s.t.} X^n(\infty)$ for all $n\ge N$.
In particular,
\begin{equation*}
  E[\wh{X}^n(\infty)]\le  E[\wh{X}^n_U(\infty)]\ra E[\wh X_C(\infty)]<\infty \qasq n\tinf,
\end{equation*}
where the convergence follows from \cite[Theorem 1]{halfin1981heavy} and the last inequality follows from Corollary 1 in this reference.

On the other hand, Lemma \ref{LemCouplingL} gives
\begin{equation*}
  E[\wh{X}^n(\infty)]\ge n^{-1/2}E[{X}^n_A(\infty)-n]\ge n^{-1/2}E[{Z}^n_A(\infty)-n].
\end{equation*}
To estimate $E[Z^n_A(\infty)]$, let $P(Ab^n_A)$ denote the long-run fraction of customer abandonment, so that $E[Z^n_A(\infty)]=\lambda^n(1-P(Ab^n_A))$.
By
\cite[Theorem 4]{garnett2002designing}, $P(Ab^n_A)=O(n^{-1/2})$.
Therefore,
\begin{equation}
  \label{Eq:EZnI}
  1-(\lambda^n)^{-1}E[Z^n_A(\infty)]=O(n^{-1/2}),
\end{equation}
so that
\begin{equation*}
  \limsup_{n\tinf}E[|\wh{X}^n(\infty)|]<\infty,
\end{equation*}
implying the statement of the proposition.
\end{proof}

\subsection{Proof of Proposition \ref{Prop:FluidTightness}} \label{secOffered}

We start by proving that $\wt{X}^n(\infty)\wedge0\Ra 0$ in $\RR$, as $n\tinf$.
Notice that
\begin{equation*}
  \wt{X}^n(\infty)\wedge0=\wt{Z}^n(\infty)=n^{-3/4}(Z^n(\infty)-n).
\end{equation*}
Using $n\ge E[Z^n(\infty)]\ge E[Z^n_A(\infty)]$ and \eqref{Eq:EZnI}, we obtain that $E[\wt{Z}^n(\infty)]\ra0$ as $n\tinf$.
By Markov's inequality,
\begin{equation*}
  P(\wt{Z}^n(\infty)>\ep)\le \ep^{-1}E[\wt{Z}^n(\infty)],\qforallq \ep>0,
\end{equation*}
implying that $\wt{X}^n(\infty)\wedge0=\wt{Z}^n(\infty)\Ra0$ in $\RR$.

Let $w^n_v$ be the {\em offered} waiting-time process in the $n$th system, namely, $w^n_v(t)$ is the time that an infinite-patient
customer (that does not abandon) would have to wait if he arrives at time $t$.
Similar to the proof of Theorem \ref{ThErgodicity} and Proposition \ref{PropSL}, $
(X^n,Q^n,Z^n,L^n,w^n_v)$ has a unique joint stationary distribution, and we let $w^n_v(\infty)$ follow the marginal stationary distribution of $w^n_v$.

Consider a generic customer with service requirement $S$, arriving to the system in steady state.
Then the offered waiting time of such a customer is independent of $S$, and is distributed like $w^n_v(\infty)$ due to PASTA.
Therefore, a generic customer in steady state enters service if and only if $S\ge \theta w^n_v(\infty)$,
and the contribution to the workload of such a customer is $S\:1\{S\ge \theta w^n_v(\infty)\}$.
In particular, the Poisson arrivals contribute to the workload of the system by $\lambda^nS\:1\{S\ge \theta w^n_v(\infty)\}$.
On the other hand, each working server reduces the workload at a constant rate 1,
so that the pool of servers reduces the workload by $Z^n(\infty)$ per unit time in steady state.
Since the mean workload is constant in steady state, we have
\begin{equation*}
E[ \lambda^nS\:1\{S \ge \theta w^n_v(\infty)\}]=E [Z^n(\infty)]\ge E [Z^n_A(\infty)],
\end{equation*}
where the inequality follows Lemma \ref{LemCouplingL} and the fact that $Z^n(\infty) = X^n(\infty) \wedge n$ and $Z^n_A(\infty) = X^n_A(\infty) \wedge n$.

Since $S$ is exponentially distributed, and is independent of $w^n_v(\infty)$,
\begin{equation*}
  E [\lambda^nS\:1\{S\ge \theta w^n_v(\infty)\}]=\lambda^nE[\int^{+\infty}_{\theta w^n_v(\infty)}se^{-s}ds]=E [e^{-\theta w^n_v(\infty)}(1+\theta w^n_v(\infty))].
\end{equation*}
Hence, for
\bequ \label{f2}
f(x) := e^{-x}(1+x), \quad x\in\RR_+,
\eeq
it holds that
\begin{equation*} 
E[ f(\theta w^n_v(\infty))]=E [S\:1\{S\ge \theta w^n_v(\infty)\}] \ge (\lambda^n)^{-1}E [Z^n_A(\infty)].
\end{equation*}
Using \eqref{Eq:EZnI}, for $f$ in \eqref{f2}.
\begin{equation}
  \label{Eq:Balance}
  1-E[ f(\theta w^n_v(\infty))]=O(n^{-1/2}).
\end{equation}


To prove the tightness of $\wt{X}^n(\infty)$, consider the steady-state probability of abandonment.
On the one hand, a generic customer abandons the system if $S\le w^n_v(\infty)$, where $S$ again stands for her service time;
On the other hand, the exponential distribution of patience time implies that the steady-state abandonment rate is $\theta E[Q^n(\infty)]$.
Therefore we have
\begin{equation*}
  \theta E[Q^n(\infty)]=\lambda^n P(S\le w^n_v(\infty))=\lambda^n(1-E[e^{-\theta w^n_v(\infty)}]).
\end{equation*}
Notice that $e^{-x}\ge (1-x)^+$ holds for any $x\in\RR_+$.
Taking $x=\theta w^n_v(\infty)$ we obtain
\begin{align*}
   E[e^{-\theta w^n_v(\infty)}]
  \ge&E[(1-\theta w^n_v(\infty))^+],
\end{align*}
so that, using $(1-x)^+=1-(x\wedge 1)$, for all $x\in\RR$,
\begin{equation}
  \label{Eq:EQUB}
  E[Q^n(\infty)]\le \theta^{-1}\lambda^nE[(\theta w^n_v(\infty)\wedge1)].
\end{equation}

To bound the right-hand side of \eqref{Eq:EQUB} from above, we elaborate on \eqref{Eq:Balance}:
Since $f'(x)=-xe^{-x}$ and $f''(x)=(x-1)e^{- x}$, for $f$ in \eqref{f2}, $f$ is strictly decreasing and concave on $[0,1]$.
Therefore
\begin{align} \label{Jensen}
 E\big[f(\theta w^n_v(\infty))\big]\le E[f\big( (\theta w^n_v(\infty))\wedge1\big)]\le f(E[(\theta w^n_v(\infty))\wedge1]).
\end{align}
where the first inequality follows the monotonicity of $f$, and the second inequality follows from Jensen's inequality.
Finally, $f(0)=1$ and $f'(x)\le -e^{-1}x$ for $x\le1$, implying that
$$f(x) = f(0) + \int_0^x f'(u)du \le 1-x^2/(2e) \qforq x\in[0,1].$$
Using the latter inequality in \eqref{Jensen} with $x=(\theta w^n_v(\infty))\wedge 1$, gives
\begin{equation*}
  E\big[f(\theta w^n_v(\infty))\big]\le f(E[(\theta w^n_v(\infty))\wedge1])\le 1-(2e)^{-1}(E[(\theta w^n_v(\infty))\wedge1])^2,
\end{equation*}
so that
\begin{equation*}
  E [(\theta w^n_v(\infty))\wedge1] \le \sqrt{\left(2e-2eE\big[f(\theta w^n_v(\infty))\big]\right)}.
\end{equation*}
It follows from \eqref{Eq:EQUB} and \eqref{Eq:Balance}, that
\begin{equation*}
  E[Q^n(\infty)]\le \sqrt{2e}\theta^{-1}\lambda^n\left(1-E\big[f(\theta w^n_v(\infty))\big]\right)^{1/2}=O(n^{3/4}).
\end{equation*}
Finally, by Markov's inequality, we have that, for any $M > 0$,
\begin{equation*}
  \limsup_{n\tinf}P(n^{-3/4}Q^n(\infty)\ge M)\le M^{-1}\limsup_{n\tinf}n^{-3/4}E[Q^n(\infty)]<\infty,
\end{equation*}
implying that $\{\wt{Q}^n(\infty) : n \ge 1\}$ is tight in $\RR$, from which the tightness of $\{\wt{X}^n(\infty) : n \ge 1\}$ follows.
\hfill\qed

\subsection{Proof of Proposition \ref{PropL}} \label{secProofPropL}

\begin{proof}[Proof of Assertion (a)]
We will show that
\begin{equation}
  \label{Eq:Linfty}
  E[L^n(\infty)]\le \lambda^n(1+\theta^{-2})(1-E[f(w^n_v(\infty))]).
\end{equation}
Together with \eqref{Eq:Balance}, this implies that $E[L^n(\infty)]=O(n^{1/2})$, which is equivalent to the statement of the assertion.

To prove \eqref{Eq:Linfty}, we consider the generalization of Little's law, known as ``$H=\lambda G$''; e.g., see \cite[Chapter 5]{wolff1989stochastic}.
Assume that the system is initialized in steady state, and let
$E^n_j$, $v^n_j$, and $T^n_j$ be, respectively, the arrival time, offered wait, and the patience of the $j$th arrival.
Also let
\begin{align*}
  g^n_j(t) &:= (t-E^n_j)1\{t\in[E^n_j,E^n_j+(v^n_j\wedge T^n_j)]\}\nonumber\\
  & \quad +(\theta v^n_j-(t-E^n_j-v^n_j))1\{t\in[E^n_j+v^n_j,E^n_j+(1+\theta)v^n_j],v^n_j\le T^n_j\},
\end{align*}
We claim that
\begin{equation}
  \label{Eq:Lt}
  L^n(t)=\sum^\infty_{j=0}g^n_j(t),\mbox{ for all }t\ge0.
\end{equation}
To see this, recall that $L^n(t)$ is the sum of the elapsed waiting time for all customers that are in the queue,
plus the remaining phase-$1$ service time for all customers in service.
Now, customer $j$ is in the queue at time $t$ if $j$ is an element of the set
$\{j:E^n_j\le t\le E^n_j+(v^n_j\wedge T^n_j)\}$, and the elapsed waiting time of that customer is $t-E^n_j$.
On the other hand, customer $j$ is in phase-$1$ of service if $j$ is an element of the set
$\{j:T^n_j\ge v^n_j,E^n_j+v^n_j\le t\le E^n_j+(1+\theta)v^n_j\}$, and the remaining phase-$1$ service time for that customer is $\theta v^n_j-(t-E^n_j-v^n_j)$.
Hence, we obtain \eqref{Eq:Lt}.

Let
\begin{equation*}
  G^n_j:=\int^\infty_0g^n_j(t)dt=(v^n_j\wedge T^n_j)^2/2+1\{T^n_j\ge v^n_j\}(\theta v^n_j)^2/2.
\end{equation*}
Since the system is considered to be in steady state, $G^n_j$ is, for each $j \ge 1$, distributed like
\begin{equation*}
G^n := (w^n_v(\infty)\wedge T)^2/2+1\{w^n_v(\infty)\le T\}(\theta w^n_v(\infty))^2/2,
\end{equation*}
where $w^n_v(\infty)$ is the stationary offered wait defined in Section \ref{secOffered},
and $T$ is an exponentially distributed random variable with rate $\theta$ that is independent of $w^n_v(\infty)$.

It follows from the following inequality
\begin{equation} \label{Gineq}
 G^n = (w^n_v(\infty)\wedge T)^2/2+1\{w^n_v(\infty)\le T\}(\theta w^n_v(\infty))^2/2\le \frac{1+\theta^2}{2}(w^n_v(\infty)\wedge T)^2,
\end{equation}
and the trivial inequality $w^n_v(\infty)\wedge T\le T$, that $E[G^n] < \infty$.
It is also easy to check that (198) in Chapter 5 of \cite{wolff1989stochastic} holds, so that, by Theorem 5 in this reference,
\begin{equation*}
E[L^n(\infty)]=\lambda^n E[G^n],
\end{equation*}
which together with the inequality in \eqref{Gineq}, gives that
\begin{equation}
  \label{Eq:LUB}
  E[L^n(\infty)]\le\lambda^n (1+\theta^2)/2E\left[(w^n_v(\infty)\wedge T)^2]\le\lambda^n(1+\theta^2)/2 E[(w^n_v(\infty)\wedge T)^2 \right].
\end{equation}
Finally, 
\begin{align*}
  E[(w^n_v(\infty)\wedge T)^2]=&E\left[\int^\infty_{w^n_v(\infty)}(w^n_v(\infty))^2\theta e^{-\theta t}dt+\int^{w^n_v(\infty)}_0t^2\theta e^{-\theta t}dt\right]\\
  =&2\theta^{-2}(1-E[(1+\theta w^n_v(\infty))e^{-\theta w^n_v(\infty)}])\\
  =&2\theta^{-2}(1-E[f(w^n_v(\infty))]).
\end{align*}
Plugging the latter equality in \eqref{Eq:LUB} gives \eqref{Eq:Linfty}, and so Assertion (a) follows from \eqref{Eq:Balance}, as was mentioned above.
\end{proof}

\begin{proof}[Proof of Assertion (b)]
By \eqref{Eq:LUB}, it is sufficient to prove that, if $\beta > 0$, then
\begin{equation}
  \label{Eq:wvUB}
  E[ (w^n_v(\infty)\wedge T)^2]=O(n^{-1}).
\end{equation}
As in the proof of Proposition \ref{Prop:DiffusionTightness}, we consider a coupling of the $M/M_{pc}/n + M_{pc}$ system with an Erlang-C system
having the same arrival process and service rate $1$. In turn, the Erlang-C system can be considered to be an $M/M_{pc}/n + M_{pc}$
system with patience that is exponentially distributed with rate $0$, so that the coupling in Lemma \ref{LemCompTrick} can be applied.
Let the two coupled systems be initially empty.

Consider a customer that arrives at both systems.
Inspecting the three cases in the proof of Lemma \ref{LemCompTrick}, we immediately see that
Case 1 irrelevant because there is no abandonment in the Erlang-C system.
The proof of Case 2 in Lemma \ref{LemCompTrick} shows that the patience of the customer in the $M/M_{pc}/n+M_{pc}$
is shorter than the waiting time of that customer in the Erlang-C system.
In particular, the delay in queue of the customer is shorter in the $M/M_{pc}/n+M_{pc}$ system than in the Erlang-C system.
Finally, the proof of Case 3 in the proof of Lemma \ref{LemCompTrick} shows again that the waiting time of the customer in
$M/M_{pc}+n/M_{pc}$ system is shorter than in the Erlang-C system.
Therefore, the waiting time of any customer is smaller in the $M/M_{pc}/n+M_{pc}$ system than in the coupled Erlang-C system.
As $\beta>0$ implies that $\beta^n>0$ for sufficiently large $n$, there exists $N_0\in \ZZ_+$ such that the Erlang-C system is stable for all $n>N_0$.
In particular, for $n\ge N_0$, the stationary waiting time of the $M/M_{pc}/n+M_{pc}$ system is stochastically dominated from above
by the stationary waiting time of the Erlang-C system.

Let $w^n_U(\infty)$ denote the stationary waiting time in the Erlang-C system, and note that
the stationary waiting time of a generic customer in the $M/M_{pc}/n+M_{pc}$ is distributed like $w^n_v(\infty)\wedge T$.
Then the stochastic ordering $w^n_v(\infty)\wedge T \le_{s.t.} w^n_U(\infty)$ just argued implies that
\begin{equation} \label{wMeanOrder}
  E[ (w^n_v(\infty)\wedge T)^2]\le E[ (w^n_U(\infty))^2].
\end{equation}
Now, the sojourn time of an arriving customer to the Erlang-C system that finds $q-1$ customers in queue, $q \ge 1$,
is distributed like the sum of $q$ independent exponential variables with mean $n^{-1}$.
Letting $\{\gamma^n_i\}$ be a sequence of i.i.d.\ exponential random variables with mean $n^{-1}$, it holds that
$w^n_U(\infty) \deq \sum^{Q^n_U(\infty)}_{i=1}\gamma^n_i$, due to PASTA, so that
\begin{equation*} 
  E[ (w^n_U(\infty))^2]=E[(\sum^{Q^n_U(\infty)}_{i=1}\gamma^n_i)^2]=n^{-2}E[(Q^n_U(\infty))^2]+n^{-2}E[Q^n_U(\infty)].
\end{equation*}
Since $\{Q^n_U(\infty) : n \ge 1\}$ is a sequence of stationary queues of $M/M/n$ systems staffed according to \eqref{Square root},
we can apply the (explicit) limits for the first and second moments of the diffusion-scaled process in
\cite[Corollary 1]{halfin1981heavy}, to conclude that $E[ (w^n_U(\infty))^2] = O(n^{-1})$. Hence, \eqref{Eq:wvUB} follows from \eqref{wMeanOrder}.
\end{proof}

\section{Remaining Proofs of Lemmas in Section \ref{secAuxProof1}}
\label{SecProofsLemmas}
In this section we prove Lemmas \ref{Lem:SBQ}, \ref{Lem:SBw}, \ref{Lem:FMartingales}, and \ref{Lem:SBXZ}.

\begin{proof}[Proof of Lemma \ref{Lem:SBQ}]
We again use a coupling of the $M/M_{pc}/n+M_{pc}$ system with another queueing system, which we denote by $\UU^n$, using
the same notation as in the proof of Proposition \ref{Prop:DiffusionTightness} for the corresponding process $(X^n_U, Q^n_U, Z^n_U)$.
We take system $\UU^n$ is a degenerated $M/M_{pc}/n+M_{pc}$ system with arrival rate $\lm^n$ and service rate $1$, in which customers have infinite patience.
  For the coupling, we initialize system $\UU^n$ and the $M/M_{pc}/n+M_{pc}$ system as follows: first, we take $X^n_U(0)=X^n(0)$;
second, any initial customer in queue has the same service time in both systems;
third, any initial customer in service system $\mathcal U$ has the same remaining service time in the $M/M_{pc}/n+M_{pc}$ system.
Note that system $\UU^n$ is not an Erlang-C system, because some of the initial customers in service may be in their phase $1$.
(There is no phase-$1$ service for any of the customers that arrive after time $0$ in this system.

Using the same arguments as in the proof of Lemma \ref{LemCompTrick}, we can construct a coupling between $X^n$ and $X^n_U$
such that $X^n(t)\le X^n_U(t)$, and thus $Q^n(t)\le Q^n_U(t)$,  w.p.1 for all $t\ge0$.
Let $\wt{Q}^n_U(t):=n^{-3/4}Q^n(n^{1/4}t)$, it is sufficient to prove that $\{\wt{Q}^n_U : n \ge 1\}$ is stochastically bounded in $D$.

Let $A$ and $S$ be two unit-rate Poisson processes.
Let $Z^n_{U0}$ and $Z^n_U$ be the processes that characterize the number of customers in
phase-$1$ service and phase-$2$ service, respectively. (Recall that arrivals have only phase-$2$ service, but initial customers may have phase-$1$ service).
Let
\begin{align*}
  \wt{X}^n_U(t) &:= n^{-3/4}(X^n_U(n^{1/4}t)-n), \quad \wt{Z}^n_U(t) := n^{-3/4}(Z^n_U(n^{1/4}t)-n),\\
  \wt{Z}^n_{U0}(t) &:= n^{-3/4}\wt{Z}^n_{U0}(n^{1/4}t).
\end{align*}
Following similar arguments as in Section \ref{secMartRep}, $\wt X^n_U$ admits the following martingale representation
\begin{equation}
  \label{Eq:su1}
  \wt{X}^n_U(t)=\wt{X}^n_U(0)-\beta^n t-n^{1/4}\int^t_0\wt{Z}^n_U(s)ds+\wt{M}^n_{UA}(t)-\wt{M}^n_{US}(t),\qforallq t\ge0,
\end{equation}
where
\begin{align*}
  \wt{M}^n_{UA}(t)&=n^{-3/4}(A(n^{1/4}\lambda^nt)-n^{1/4}\lambda^nt),\qandq \\
  \wt{M}^n_{US}(t)&=n^{-3/4}\Big(S\big(\int^{n^{1/4}t}_0Z^n_U(s)ds\big)-\int^{n^{1/4}t}_0Z^n_U(s)ds\Big),\qforq t\ge0.
\end{align*}
It follows from the Poisson FCLT (e.g., Theorem 4.2 in \cite{pang2007martingale}) that
\begin{equation*}
  n^{-5/8}(A(n^{5/4}\cdot)-n^{5/4}\eta(\cdot))\Ra B(\cdot)\qandq n^{-5/8}(S(n^{5/4}\cdot)-n^{5/4}\eta(\cdot))\Ra B(\cdot),\qforallq t\ge0,
\end{equation*}
for a standard Brownian motion $B$, so that
\begin{equation*}
  \|n^{-3/4}(A-\eta)\|_{n^{5/4}t}\Ra 0\qandq \|n^{-3/4}(S-\eta)\|_{n^{5/4}t}\Ra0,\qforallq t\ge0,
\end{equation*}
Therefore
\begin{equation*}
  n^{1/4}\lambda^nt=O(n^{5/4})t\qandq\int^{n^{1/4}t}_0Z^n_U(s)ds\le n^{5/4}t,
\end{equation*}
imply that
\begin{equation*}
  (\wt{M}^n_{UA},\wt{M}^n_{US})\Ra (0\eta,0\eta)\qinq D^2,\qasq n\tinf.
\end{equation*}
Consider the process
\begin{equation*}
  \xi^n(t) := \wt{M}^n_{UA}(t)-\wt{M}^n_{US}(t)+n^{1/4}\int^t_0\wt{Z}^n_{U0}(s)ds, \quad t \ge 0.
\end{equation*}
Using similar arguments as in the proof of Proposition \ref{Prop:F}(a), one can show that
\begin{equation*}
  \label{EqZU0}
  n^{1/4}\int^\cdot_0\wt{Z}^n_{U0}(s)ds=o_p(1),
\end{equation*}
so that $\xi^n=o_P(1)$.
Using the equality $\wt{Z}^n_U=\wt{X}^n_U\wedge0-Z^n_{U0}$, \eqref{Eq:su1} becomes
\begin{align*}
  &\wt{X}^n_U(t)=\wt{X}^n_U(s)-\beta^n (t-s)-n^{1/4}\int^t_s\wt{X}^n_U(s)\wedge0ds+\xi^n(t)-\xi^n(s), \quad t \ge s\ge 0.
\end{align*}
Take $s := \sup\{u \in[0,t]: X^n_U(u)<0\}$, where, for $\emptyset$ denoting the empty set, $\sup \emptyset := 0$.
Then either $s=t$ or $X^n_U\ge0$ on $[s,t)$, implying that $n^{1/4}\int^t_s\wt{X}^n_U(s)\wedge0ds=0$.
Moreover, either $s=0$ or $X^n_U(s-)\le0$, where in the latter case $X^n_U(s)\le 1$ since $X^n_U$ only has either positive or negative jumps
of size $1$.
In particular, $X^n_U(s)\le X^n_U(0)\vee0+1$. Therefore,
\begin{equation*}
  \wt{X}^n_U(t)\le \wt{X}^n_U(0)\vee0+n^{-3/4}-(\beta^n\wedge 0) t+2\|\xi^n\|_t.
\end{equation*}
Notice that the right-hand side is strictly positive and  non-decreasing in $t$.Using $\wt{Q}^n_U=\wt{X}^n_U\vee0$,
\begin{equation*}
  \|\wt{Q}^n_U\|_t\le \wt{X}^n_U(0)\vee0+n^{-3/4}-(\beta^n\wedge 0) t+2\|\xi^n\|_t.
\end{equation*}
As $n\tinf$, $\wt{X}^n_U(0)\Ra X_0$ in $\RR$ and $\beta^n\ra\beta\in\RR$, so that the right-hand side is stochastically bounded in $\RR$ for any $t\ge0$,
so that $\{\wt Q^n : n \ge 1\}$ is stochastically bounded in $D$, as stated.
\end{proof}

\begin{proof}[Proof of Lemma \ref{Lem:SBw}]
We first observe that, for any $t > 0$,
\begin{equation}
  \label{Eq:SBw}
  Q^n(t)\ge A^n(t)-A^n(t-w^n(t))-R^n(t)+R^n(t-w^n(t)).
\end{equation}
To see this, note that the head-of-line customer arrived at time $t-w^n(t)$.
Thus, any waiting customer at time $t$ must either be an initial customer, or a customer that
arrived to the system during $[t-w^n(t),t)$. 
On the other hand, the number of those customers that arrived during $[t-w^n(t),t]$ and abandoned by time $t$ is clearly no larger than the total
 number of abandonments during $[t-w^n(t),t]$. Thus, we get \eqref{Eq:SBw}. 

Notice that
\begin{align*}
  n^{-3/4}A^n(n^{1/4}t)&=\wt{M}^n_A(t)+n^{-1/2}\lm^n t,\\
  n^{-3/4}A^n(n^{1/4}t-w^n(n^{1/4}t))&=\wt{M}^n_A(t-n^{-1/2}\wt{w}^n(t))+n^{-1/2}\lm^n(t-n^{-1/2}\wt{w}^n(t))\\
  n^{-3/4}R^n(n^{1/4}t)&=\wt{M}^n_R(t)+\theta\int^t_0n^{1/4}\wt{Q}^n(s)ds\\
  n^{-3/4}R^n(n^{1/4}t-w^n(n^{1/4}t))&=\wt{M}^n_R(t-n^{-1/2}\wt{w}^n(t))+\theta\int^{t-n^{-1/2}\wt{w}^n(t)}_0n^{1/4}\wt{Q}^n(s)ds,
\end{align*}
Plugging these equalities in \eqref{Eq:SBw} and using the LOF scaling gives
\begin{align*}
  \wt{Q}^n(t)\ge&n^{-3/4}\big(A^n(n^{1/4}t)-A^n(n^{1/4}t-w^n(n^{1/4}t))\\
  &-R^n(n^{1/4}t)+R^n(n^{1/4}t-w^n(n^{1/4}t))\big)\\
  = &n^{-1}\lambda^n\wt{w}^n(t)+\wt{M}^n_A(t)-\wt{M}^n_A(t-n^{-1/2}\wt{w}^n(t))\\
  &-\theta\int^t_{t-n^{-1/2}\wt{w}^n(t)}n^{1/4}\wt{Q}^n(s)ds-\wt{M}^n_R(t)+\wt{M}^n_R(t-n^{-1/2}\wt{w}^n(t))\\
  \ge &\wt{w}^n(t)(n^{-1}\lambda^n-n^{-1/4}\theta\|\wt{Q}^n\|_t)-2\|\wt{M}_A^n\|_t-2\|\wt{M}_R^n|_t.
\end{align*}
The statement follows from the facts that the processes $\wt{Q}^n$, $\wt{M}^n_A$, and $\wt{M}^n_R$ are all $o_P(1)$,
$n^{-1/4}\|\wt{Q}^n\|_t\Ra0$ in $\RR$, and $\lm^n/n \ra 1$.
\end{proof}

\begin{proof}[Proof of Lemma \ref{Lem:FMartingales}]
Fix $n \ge 1$. That $F^n(s,t)$ is integrable follows from $-\theta^{-1}\lambda^n\le F^n(s,t)\le A^n(s)$.
Let $\{\mathcal{G}^n_{s,t}:s\ge0\}$ be the natural filtration generated by $F^n(s,t)$, augmented by including all $P$-null sets.
Note that, for $0\le s_1<s_2\le t$,
\begin{align*}
  F^n(s_2,t)-F^n(s_1,t)=&\int^{s_2}_{s_1}1\{E^n_{A^n(s)}+T^n_{A^n(s)}>t\}dA^n(s) 
  -\theta^{-1}\lambda^n(e^{\theta(s_2-t)}-e^{\theta(s_1-t)}).
\end{align*}
and that the right-hand side is independent of $\mathcal{G}^n_{s_1,t}$. Hence,
\begin{align*}
  &E\Big[\int^{s_2}_{s_1}1\{E^n_{A^n(s)}+T^n_{A^n(s)}>t\}dA^n(s)\Big| \mathcal{G}^n_{s_1,t}\Big]\\
  =&E\Big[\int^{s_2}_{s_1}1\{E^n_{A^n(s)}+T^n_{A^n(s)}>t\}dA^n(s)\Big]\\
  =&E\Big[\int^{s_2}_{s_1}E\big[1\{E^n_{A^n(s)}+T^n_{A^n(s)}>t\}\big| E^n_{A^n(s)}\big]dA^n(s)\Big]
\end{align*}
Since $T^n_{A^n(s)}$---the patience of the last customer to arrive before time $s$---is exponentially distributed and is independent of the arrival time $E^n_{A^n(s)}$,
\begin{equation*}
  E\big[1\{E^n_{A^n(s)}+T^n_{A^n(s)}>t\}\big| E^n_{A^n(s)}\big]=\exp(-\theta(E^n_{A^n(s)}-t))
\end{equation*}
Therefore,
\begin{align*}
  &E\Big[\int^{s_2}_{s_1}E\big[1\{E^n_{A^n(s)}+T^n_{A^n(s)}>t\}\big| E^n_{A^n(s)}\big]dA^n(s)\Big] = E\Big[\int^{s_2}_{s_1}\exp(-\theta(E^n_{A^n(s)}-t))dA^n(s)\Big].
\end{align*}
Finally, $A^n$ is a simple counting process, and $E^n_{A^n(s)}=s$ when $dA^n(s) = 1$, so that
\begin{align*}
  E\Big[\int^{s_2}_{s_1}\exp(-\theta(E^n_{A^n(s)}-t))dA^n(s)\Big]
  &=E\Big[\int^{s_2}_{s_1}e^{-\theta(s-t)}dA^n(s)\Big]\\
  &=\theta^{-1}\lambda^n(e^{\theta(s_2-t)}-e^{\theta(s_1-t)}).
\end{align*}
Thus,
\begin{equation*}
  E[F^n(s_2,t)-F^n(s_1,t)|\mathcal{G}_{s_1,t}]=0,\qforq 0\le s_1\le s_2\le t,
\end{equation*}
implying that $\{F^n(s,t):s\in [0,t]\}$ is a martingale.

To prove that $\{e^{\theta t}\sup_{s\in[0,t]}|F^n(s,t)|:t\ge0\}$ is a submartingale,
let $\{\mathcal{G}^n_t:t\ge0\}$ be the right-continuous filtration generated by
\begin{equation*}
\left(A^n(s), 1\{T^n_k+E^n_k< t\}:s\le t, k\le A^n(t)\right),
\end{equation*}
and augmented by including all $P$-null sets.
It is easy to check that $F^n(s,s+t)\in\mathcal{G}^n_{s+t}$ and $\sup_{s\in[0,t]}|F(s,t)|\in\mathcal{G}^n_t$.
We will show that $\sup_{s\in[0,t]}|F(s,t)|$ is a $\mathcal{G}^n$-submartingale, and therefore also a submartingale with respected to the (augmented) natural filtration.

To this end, fix $0 \le t_1\le t_2$.
for $s_1\in[0,t_1]$ such that $E^n_{A^n(s_1)}+T^n_{A^n(s_1)}>t_1$. Due to the memoryless property of $T^n_{A^n(s_1)}$,
$E^n_{A^n(s_1)}+T^n_{A^n(s_1)}-t_1$ is also an exponential random variable with rate $\theta$, so that
\begin{equation*}
  P\left(E^n_{A^n(s_1)}+T^n_{A^n(s_1)}>t_2|E^n_{A^n(s_1)}+T^n_{A^n(s_1)}>t_1\right)=e^{\theta(t_1-t_2)}.
\end{equation*}
Trivially,
\begin{equation*}
  E^n_{A^n(s_1)}+T^n_{A^n(s_1)}\le t_2 \qifq ~ E^n_{A^n(s_1)}+T^n_{A^n(s_1)}\le t_1.
\end{equation*}
Now, $E^n_{A^n(s_1)}+T^n_{A^n(s_1)}>t_2$ is, condition on the event $\{E^n_{A^n(s_1)}+T^n_{A^n(s_1)}>t_1\}$, independent of $\mathcal{G}^n_{t_1}$.
Finally, $1\{E^n_{A^n(s_1)}+T^n_{A^n(s_1)}>t_1\}\in \mathcal{G}^n_{t_1}$, implying that
\begin{align*}
  E\left[1\{E^n_{A^n(s_1)}+T^n_{A^n(s_1)}>t_2\}|\mathcal{G}^n_{t_1}\right]
  =e^{\theta(t_1-t_2)}1\{E^n_{A^n(s_1)}+T^n_{A^n(s_1)}>t_1\}.
\end{align*}
Integrating both sides of the equality with respect to $s_1$ over $[0,s]$ and using the equality

\begin{equation*}
  e^{\theta(t_1-t_2)}\theta^{-1}\lambda^n(e^{-\theta (t_1-s)}-e^{-\theta t_1})=\theta^{-1}\lambda^n(e^{-\theta (t_2-s)}-e^{-\theta t_2}),
\end{equation*}
gives
\begin{equation*}
  E \left[F^n(s,t_2)|\mathcal{G}^n_{t_1}\right]=e^{\theta(t_1-t_2)}F^n(s,t_1),\qforallq 0\le s\le t_1\le t_2.
\end{equation*}
In particular, $\{e^{\theta(s+t)}F^n(s,s+t):t\ge0\}$ is a $\{\mathcal{G}^n_{s+t}:t\ge0\}$-martingale.

Now let $0\le t_1\le t_2$ and an arbitrary random time $S\le t_1$ such that $S\in\mathcal{G}^n_{t_1}$, we have
\begin{align*}
  E \Big[e^{\theta t_2}\sup_{s\in[0,t_2]}|F^n(s,t_2)|\Big| \mathcal{G}^n_{t_1}\Big]
  \ge E \Big[|e^{\theta t_2}F^n(S,t_2)|\Big| \mathcal{G}^n_{t_1}\Big].
\end{align*}
It follows from the facts that $S\in\mathcal{G}^n_{t_1}$, and that $\{e^{\theta(s+t)}F^n(s,s+t):t\ge0\}$ is a $\{\mathcal{G}^n_{s+t}:t\ge0\}$-martingale, that
\begin{align*}
 E \left[e^{\theta t_2}F^n(S,t_2)\big| \mathcal{G}^n_{t_1}\right]= e^{\theta t_1}F^n(S,t_1).
\end{align*}
By Jensen's inequality
\begin{align*}
  E \left[|e^{\theta t_2}F^n(S,t_2)|\big| \mathcal{G}^n_{t_1}\right]
  \ge e^{\theta t_1}|F^n(S,t_1)|,
\end{align*}
so that
\begin{equation} \label{ineqS}
  E \Big[e^{\theta t_2}\sup_{s\in[0,t_2]}|F^n(s,t_2)|\Big| \mathcal{G}^n_{t_1}\Big]
  \ge e^{\theta t_1}|F^n(S,t_1)|.
\end{equation}

Finally, $|F^n(s,t_1)|\in\mathcal{G}^n_{t_1}$ for any $s\in[0,t_1]$. Since $F^n(\cdot\wedge t_1,t_1)$ has right-continuous paths, for each $\ep>0$ we can choose $S_\ep\in\mathcal{G}^n_{t_1}$ such that
\begin{equation*}
  \sup_{s\in[0,t_1]}|F^n(s,t_1)|\le |F^n(S_\ep,t_1)|+\ep, w.p.1.
\end{equation*}
Taking $S=S_\ep$ in \eqref{ineqS} gives
\begin{equation*}
  E \Big[e^{\theta t_2}\sup_{s\in[0,t_2]}|F^n(s,t_2)|\:\Big| \mathcal{G}^n_{t_1}\Big]\ge e^{\theta t_1}\sup_{s\in[0,t_1]}|F^n(s,t_1)|-\ep,\mbox{ for all }t_1\ge0\qandq \ep>0.
\end{equation*}
The proof follows upon taking $\ep\ra0$.
\end{proof}

\begin{proof}[Proof of Lemma \ref{Lem:SBXZ}]
We first prove that $\{\wt{U}^n_1 : n \ge 1\}$ is stochastically bounded in $\D$.
Consider the LOF-scaled process in \eqref{Eq:U1},
\begin{equation*}
  \wt{U}^n_1(t)=\int^t_{\wt{T}^n_0\wedge t}1\{n^{-1/2}\theta\wt{w}^n(s-)+s> t\}d\wt{M}^n_S(s).
\end{equation*}
Notice that the integrand is non-negative and satisfies

\begin{equation}
  \label{Eq:Z1Indicator}
  1\{n^{-1/2}\theta\wt{w}^n(s-)+s> t\}\le 1\{s\ge t-n^{-1/2}\theta\|\wt{w}^n\|_t\},\mbox{ for all }0\le s\le t.
\end{equation}
Thus,
\begin{align*}
  |\wt{U}^n_1(t)| & \le \int^t_{0}1\{s\ge t-n^{-1/2}\theta\|\wt{w}^n\|_t\}d|\wt{M}^n_S(t)| \\
  &\le\int^t_{t-n^{-1/2}\theta\|\wt{w}^n\|_t}d|\wt{M}^n_S(t)|\\
  &\le 2\|\wt{M}^n_S\|_t+2\theta\|\wt{w}^n\|_t,
\end{align*}
where the last inequality follows from \eqref{Eq:SBU1}.
By Proposition \ref{Prop:F}(b) and Lemma \ref{Lem:SBw}, the right-hand side is stochastically bounded in $\D$,
implying that $\{\wt{U}^n_1 : n \ge 1\}$ is stochastically bounded in $\D$ as well.

To prove that $\{\wt{Z}^n_1-\wt{Z}^n_0 : n \ge 1\}$ is stochastically bounded in $D$, consider the LOF-scaled process in \eqref{Eq:ssc}
\begin{align*}
  \wt{Z}^n_1(t)-\wt{Z}^n_0(t)=&\int^t_{\wt{T}^n_0\wedge t}1\{n^{-1/2}\theta\wt{w}^n(s-)+s> t\}(n^{1/2}+n^{1/4}\wt{Z}^n_2(s))ds+\wt{U}_1^n(t).
\end{align*}
Again, using $\wt{Z}^n_2\le0$ w.p.1 and \eqref{Eq:Z1Indicator},
\begin{align*}
   |\wt{Z}^n_1(t)-\wt{Z}^n_0(t)| \le&n^{1/2}\int^t_01\{s\ge t-n^{-1/2}\theta\|\wt{w}^n\|_t\}ds+|\wt{U}_1^n(t)|\le\theta\|\wt{w}^n\|_t+\|\wt{U}_1\|_t.
\end{align*}
Since the right-hand side is non-decreasing in $t$,
Lemma \ref{Lem:SBw} and the stochastic boundedness of $\{\wt{U}^n : n \ge 1\}$ in $\D$
imply that $\{\wt{Z}^n_1-\wt{Z}^n_0 : n \ge 1\}$ is also stochastically bounded in $\D$.

To prove that $\{\wt{X}^n : n \ge 1\}$ is stochastically bounded in $\D$,
we use the same arguments as in the proof of Proposition \ref{Prop:Tightness} to obtain \eqref{EqXve}, and in particular,
\begin{equation*}
  \wt{X}^n \ge \xi^n + \varepsilon^n,
\end{equation*}
where $\varepsilon^n=o_P(1)$.
It follows from the stochastic boundedness of $\{\wt{w}^n : n \ge 1\}$ and $\{\wt{Z}^n_1-\wt{Z}^n_0 : n \ge 1\}$ in $D$,
that $\xi^n$ is also $O_P(1)$.
Therefore, $\xi^n+\varepsilon^n\le \wt{X}^n\le\wt{Q}^n$ implies that $\{\wt{X}^n\}$ is stochastically bounded in $\D$,
and thus $\wt{Z}^n=\wt{X}^n\wedge0$, implies that $\{\wt{Z}^n : n \ge 1\}$ is stochastically bounded in $D$.
Finally, $\wt{Z}^n_2+\wt{Z}^n_0=\wt{Z}^n-(\wt{Z}^n_1-\wt{Z}^n_0)$ implies that $\{\wt{Z}^n_2+\wt{Z}^n_0 : n \ge 1\}$ is stochastically bounded in $D$.
\end{proof}

 \bibliographystyle{imsart-number}
\bibliography{Diff_corr_cite}

\end{document}